\documentclass[12pt,a4paper]{amsart}
\usepackage{amsthm}
\usepackage{amsmath,amscd,upref}

\newtheorem{Theorem}{Theorem}[section]
\newtheorem{Proposition}[Theorem]{Proposition}
\newtheorem{Lemma}[Theorem]{Lemma}
\newtheorem{Corollary}[Theorem]{Corollary}

\theoremstyle{definition}

 \theoremstyle{remark}

\makeatletter
       \def\@makefnmark{%
               \leavevmode
               \raise.9ex\hbox{\check@mathfonts
                       \fontsize\sf@size\z@\normalfont%
                               \@thefnmark}%
       }
       \makeatother

\makeatletter
    
    \@addtoreset{equation}{section}
  \makeatother
\def\ts{\hskip 0.4mm}
\def\tts{\hskip 0.2mm}
\def\nts{\hskip-0.2mm}
\newfont{\BBb}{msbm10 scaled\magstep1}
\newfont{\sBbb}{msbm7 scaled\magstep1}

\newcommand{\ZZ}{\mbox{\BBb Z}}
\newcommand{\sZZ}{\mbox{\sBbb Z}}

\def\Sum{\textstyle\sum\limits}
\def\Prod{\textstyle\prod\limits}

\def\Frac#1#2{\displaystyle\frac{\mathstrut\raisebox{-.3ex}{$\;#1\;$}}{\;#2\;}}
\def\Sum{\displaystyle\sum}
\def\wtilde{\widetilde}
\def\ts{\hskip 0mm}

\DeclareMathOperator{\Img}{Im}

\def\mathvskip{\vskip\abovedisplayskip}

\def\trdeg{{\rm tr\,deg}}
\thanks{Partially supported by the Grant-in-Aid for Scientific Research (C) 
21540104, Japan Society for the Promotion of Science.
}
\title[Cohomology mod 3 of the classifying space of $E_6$, II]%
{Cohomology mod 3 of the classifying space of the exceptional Lie group $E_6$, II : %
The Weyl group invariants
}

\author[M.Mimura]{Mamoru Mimura}
\address{Department of Mathematics
Faculty of Science Okayama University;
3-1-1 Tsushima-Naka, Okayama
700-8530, Japan}
\email{mimura@math.okayama-u.ac.jp}
\author[Y.Sambe]{Yuriko Sambe}
\address{
School of Business Administration
Senshu University;
2-1-1 Tama-ku, Kawasaki
214-8580, Japan}
\email{thm0166@isc.senshu-u.ac.jp}
\author[M.Tezuka]{Michishige Tezuka}
\address{Department of Mathematical Sciences
Faculty of Science University of Ryukyus; 
Nishihara-cho, Okinawa 903-0213, Japan}
\email{tez@math.u-ryukyu.ac.jp}
\begin{document}
\maketitle

\begin{abstract}
We calculate the Weyl group invariants with respect to 
a maximal torus of the exceptional Lie group $E_6$.
\vspace{4pt}
\vspace{4pt}\\
Key words and phrases, exceptional Lie group, Weyl group
\end{abstract}

\section{Introduction.}
     The present paper is a sequel to \cite{MSTe2}.

     Let  $E_6$  be the compact,  simply connected, 
exceptional  Lie group of rank  6,  
$T$  its maximal torus,  $BX$  the classifying space of  $X$
for  $X = E_6$ or $T$,  and  $W(E_6)$  the Weyl group 
which acts on $H^*(BT\ts ;\ZZ)$, and hence on $H^*(BT\ts ;\ZZ_3)$, as usual.
As is well known,  the invariant subalgebra
$H^*(BT\ts;\ZZ_3)^{W(E_6)}$   contains the image 
of the homomorphism $B\ts i\ts^* : H^*(BE_6\ts;\ZZ_3) \to H^*(BT\ts;\ZZ_3)$.

     The Rothenberg\ts-Steenrod spectral sequence %
$\{E_r(X),d\tts_r\}$  for $X$  a
compact associative  $H$\nts-space has
$$\begin{array}{llllllll}
E\tts_2(X) &= {\rm Cotor}_A(\ZZ_p,\ZZ_p) \quad 
                  \mbox{with}\quad A = H^*(X\ts;\ZZ_p),\\
E_{\infty}(X)\;&= GrH^*(BX\ts;\ZZ_p)
\end{array}$$
where  $p$  is a prime number.

     The  $E_2$\tts-\tts term  Cotor$_A(\ZZ_p,\ZZ_p)$  for  $X=E_6$  and  $p=3$  
is calculated in  \cite{MS} (see also [MST3]).  
In this series of papers,  
it will be shown that the spectral sequence
associated with  $E_6$  collapses 
and the ring structure of $H^*(BE_6\ts;\ZZ_3)$
will be determined  
(cf.  \cite{KM}  where some of the ring structure of
                       $H^*(BE_6\ts;\ZZ_3)$  was determined).

     The present paper is organized as follows.  
In  Section 2,  the action
of the Weyl group  $W(E_6)$  on  $H^*(BT\ts;\ZZ_3)$  
will be described in terms
of some generators of $H^*(BT\ts;\ZZ_3)$.  
From this,  6  elements in
$H^*(BT\ts;\ZZ_3)^{W(E_6)}$  are immediately obtained,  
that is,  $x_4$, $x_8$, $y_{10}$, $x_{20}$,
$y_{22}$  and  $y_{26}$.  
In  Section 3,  another set  $S$  of elements in
$H^*(BT\ts;\ZZ_3)$  which is invariant as a set under 
$W(E_6)$ is considered.
The elementary symmetric functions on the elements of  $S$  are
$W(E_6)$\ts-\ts invariant and we find another element  $x_{36}$.  
In Section 4,  it is shown firstly that
$$ H^*(BT\ts;\ZZ_3)^{W(E_6)} \subset  
           \ZZ_3(x_4)[x_8,y_{10},x_{20},y_{22},x_{36}].$$
Then  6  more generators of  $H^*(BT\ts;\ZZ_3)^{W(E_6)}$ 
are totally determined explicitly.  We prove

\bigskip
     {\bf Theorem 4.13}\quad
{\sl $H^*(BT\ts;\ZZ_3)^{W(E_6)}$  is generated 
by the following thirteen elements\ts:}
$$x_4,\  x_8,\  y_{10},\  x_{20},\  y_{22},\  y_{26},\  
x_{36},\  x_{48},\  x_{54},\  y_{58},\  y_{60},\  y_{64},\  y_{76}.$$

We thank T. Torii for suggesting us appropriate references on Galois theory and M. Nakagawa for pointing a fact stated in \cite{LS}. 
We also thank A. Kono for reading the manuscript and giving us
valuable suggestions. Last but not least, we thank Prof. H. Toda for tutoring us the relation between the Dynkin diagram and the generators of the Weyl group. 

This paper has taken a long time to write. 
In fact, the results of the paper were announced in [T].

\section{Some invariant elements in low degrees.}
  Let  \ts$T \subset E_6$  be a fixed maximal torus of $E_6$,
         \ts$V_T$  the universal covering of \ts$T$,
          \ts$V_T^*$  \ts the dual of  \ts$V_T$.
According to  Bourbaki \cite{B},  the Dynkin diagram is 
given as follows\ts:
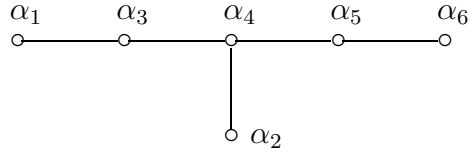
$$
\setlength{\unitlength}{2pt}
\begin{picture}(130,30)(-10,-20)
\multiput(12.3, 0.3)(20,0){4}{\line(1,0){17.8}}
\put(10,-1){$\circ$}
\put(30,-1){$\circ$}
\put(50,-1){$\circ$}
\put(70,-1){$\circ$}
\put(90,-1){$\circ$}
\put(10, 4.5){$\alpha_1$}
\put(30, 4.5){$\alpha_3$}
\put(50, 4.5){$\alpha_4$}
\put(70, 4.5){$\alpha_5$}
\put(90, 4.5){$\alpha_6$}
\put(50,-19){$\circ$}
\put(55,-19){$\alpha_2$}
\put(51.4,-1.1){\line(0,-1){15.1}}
\end{picture}$$
where  \ts$\alpha\tts_i \in V_T^*$ \  for \  $i=1,2, \ldots , 6$  
are the simple roots of $E_6$.
Let $\langle\;,\,\rangle$  be the invariant metric 
on the Lie algebra of \ts$E_6$, $V_T$  and  $V_T^*$,  
normalized in such a way that
$$\begin{array}{llllllll}
 \smallskip\langle\alpha\tts_i,\alpha\tts_i\rangle &= 2,        & \\
 \smallskip\langle\alpha\tts_i,\alpha_j    \rangle &= -1 \qquad &
               \mbox{if }  \ i \neq j \  \mbox{ and }  \ \alpha\tts_i  
               \mbox{ and }  \alpha_j\ts  \mbox{ are connected,} \\
 \langle\alpha\tts_i,\alpha_j\rangle& = 0          & \mbox{otherwise.}
\end{array}$$

Let  $\beta_j$  be the corresponding fundamental weights\ts:
               $\langle \, \beta_j,\alpha\tts_i \rangle = \delta_{ij}$.
Let us denote by  $R\tts_i$  the reflection with respect to the
               hyperplane  \  $\alpha\tts_i = 0$.  Then
$$R\tts_i(\beta\tts_i) = \beta\tts_i 
   - \Sum_j \langle \alpha\tts_i, \alpha_j \rangle \, \beta_j
   \quad\mbox{and}\quad   R\tts_i(\beta_j) = \beta_j
   \enskip\mbox{ for }\enskip i \neq j.
$$

Denote by $U$  the centralizer of the torus  $T\tts^1$  defined by
          $\langle\alpha\tts_i,t \ts\rangle = 0$  for  $i=2,3,\ldots ,6$  
and  $t\in T$.
Then $U$  is a closed connected subgroup of maximal rank and of local type
$D_5\times T^1$  such that  $D_5\cap T^1 = \ZZ_4$ (see \ts[\ts V\ts]\ts).
The  Weyl groups  $W(E_6)$  and  $W(U)$  are generated by
  $R\tts_1$, $R\tts_2$, \ldots, $R\tts_6$  
  and  $R\tts_2$, \ldots, $R\tts_6$,  respectively.

Put
\begin{equation}\begin{array}{llllllll}
     & \tau_6&=\beta_6,&\\
     & \tau_5&=R\tts_6(\tau_6)&=\beta_5-\beta_6,\\
     & \tau_4&=R\tts_5(\tau_5)&=\beta_4-\beta_5,\\
     & \tau_3&=R\tts_4(\tau_4)&=\beta_2+\beta_3-\beta_4,\\
     & \tau_2&=R\tts_3(\tau_3)&=\beta_1+\beta_2-\beta_3,\\
     & \tau_1&=R\tts_1(\tau_2)&=-\beta_1+\beta_2,\\
     & x &=\beta_2& = {1\over3}\,(\tau_1+\tau_2+\cdots +\tau_6).
\label{eq:11}
\end{array}
\end{equation}
Then  $\beta\tts_i$  are linear combinations of \ts$\tau_j$'s  
and  $x$  as follows\ts:
$$\begin{array}{llllllll}
  &\beta_1=x-\tau_1,\quad & \beta_2=x, \quad 
              & \beta_3=-x+\tau_3+\tau_4+\tau_5+\tau_6,\\
  &\beta_4=\tau_4+\tau_5+\tau_6, \quad & \beta_5=\tau_5+\tau_6,\quad 
              & \beta_6=\tau_6.
\label{eq:11$'$}
\end{array}$$

     Further we put 

\medskip
$\qquad\quad
t=x-\tau_1\enskip\mbox{ and }\enskip 
t\tts_i=\tau_{i+1}- {\textstyle{1\over2}}\,t\enskip
\mbox{ for }\enskip  i=1,\ldots ,5.$

\medskip\noindent
Then we have
\begin{equation}
\begin{array}{lll}
  &t &=\beta_1,\\
  &t_1&= {1\over2}\,\beta_1+\beta_2-\beta_3,\\
  &t_2&=-{1\over2}\,\beta_1+\beta_2+\beta_3-\beta_4,\\
  &t_3&=-{1\over2}\,\beta_1+\beta_4-\beta_5,\\
  &t_4&=-{1\over2}\,\beta_1+\beta_5-\beta_6,\\
  &t_5&=-{1\over2}\,\beta_1+\beta_6,
\label{eq:12}
\end{array}
\end{equation}
and
\begin{equation}
\begin{array}{lll}
\hspace{22mm}
& \beta_1 &=t,\\
& \beta_2 &= {3\over4}\,t + {1\over2}\,t_1 + {1\over2}\,t_2
             + {1\over2}\,t_3 + {1\over2}\,t_4 + {1\over2}\,t_5,\\%
\hspace{-33.2mm}\lower2.0ex\hbox{(2.2)$'$}
& \beta_3 &= {5\over4}\,t - {1\over2}\,t_1 + {1\over2}\,t_2
             + {1\over2}\,t_3 + {1\over2}\,t_4 + {1\over2}\,t_5,\\[-1.6ex]
& \beta_4 &= {3\over2}\,t + t_3 + t_4 + t_5,\\
& \beta_5 &= t + t_4 + t_5,\\
& \beta_6 &= {1\over2}\,t + t_5.
\nonumber
\end{array}
\end{equation}

     Denote \  $t_1+t_2+t_3+t_4+t_5$ \ by \ $c_1$. Then we have
\begin{equation}
     c_1 = t_1 +t_2 +\cdots+t_5  
          = \tau_2 +\tau_3 +\cdots+\tau_6 - {\textstyle{5\over2}}\ts t 
          = 2x- {\textstyle{3\over2}}\ts t.
\label{eq:13}
\end{equation}

The  $R\ts_i$\ts-operations are given by
$$\baselineskip20pt
\halign{\qquad$#$\qquad\hfil &&\hfil$#$\qquad\hfil \cr
    & R\tts_1 &              R\tts_2 &   R\tts_3 &   R\tts_4 &   R\tts_5 &   R\tts_6\cr
 t  & {1\over4}\,t-t_1+{1\over2}\,c_1\cr
 t_1& -{3\over8}\,t+{1\over2}\,t_1+{1\over4}\, c_1   &  -t_2 &   t_2\cr
 t_2& {3\over8}\,t+{1\over2}\,t_1+t_2-{1\over4}\,c_1 &  -t_1 &   t_1 
                    &   t_3\cr
 t_3& {3\over8}\,t+{1\over2}\,t_1+t_3-{1\over4}\,c_1 &       &        
                    & t_2   & t_4\cr
 t_4& {3\over8}\,t+{1\over2}\,t_1+t_4-{1\over4}\,c_1 &       &        
                    &       &t_3    &t_5\cr
 t_5& {3\over8}\,t+{1\over2}\,t_1+t_5-{1\over4}\,c_1 &       &        
                    &       &       &t_4\cr
}$$
where the blanks indicate the trivial action.
     Taking coefficients in  $\ZZ_3$,  we have
\begin{equation}
\begin{array}{llll}
  & R\tts_1(t)&=t-(t_1+c_1),&\cr
  & R\tts_1(t_1)&=-t_1+c_1,&\cr
  & R\tts_1(t\tts_i)&=t\tts_i-(t_1+c_1)&\qquad\quad 
                                 \mbox{for } \ i = 2,3,4,5.
\label{eq:14}
\end{array}
\end{equation}

          Roots or weights will be considered in the usual way (cf. \cite{BH})  as
elements of  $H^*(T)$,  $H^*(T\ts;\ZZ_3)$, $H^*(BT)$  
or $H^*(BT\ts;\ZZ_3)$.
Then the weights  $\beta_1,\ldots ,\beta_6$  generate  
$H^*(BT)$  as well as  $H^*(BT\ts;\ZZ_3)$
and so do the elements  $t, t_1, \ldots , t_5$:
$$ H^*(BT\ts;\ZZ_3) = \ZZ_3[\ts t, t_1, \ldots , t_5]
\quad \mbox{with }\enskip t,t_1,t_2,\ldots,t_5 \in H^2(BT\ts;\ZZ_3).$$

    From now on until the end of Section 2 all coefficients are in  $\ZZ_3$.

\bigskip
The  Weyl group  $W(E_6)$  acts on  $BT$  and hence on  
$H^*(BT\ts;\ZZ_3)$, and the invariant subalgebra  
$H^*(BT\ts;\ZZ_3)^{W(E_6)}$  
contains the image of the homomorphism   
$B\ts i\ts^* : H^*(BE_6\ts;\ZZ_3) \to H^*(BT\ts;\ZZ_3)$  induced by the
natural map  $B\ts i : BT  \to  BE_6$.
The action of  $W(U)$  is the same as the usual action of 
$W(SO(10))$.
Thus we obtain
\begin{equation}
\begin{array}{c}
{\rm Im}\;B\ts i\ts^*(H^*(BE_6\ts;\ZZ_3)) 
                  \subset H^*(BT\ts;\ZZ_3)^{W(E_6)}\\ 
\subset H^*(BT\ts;\ZZ_3)^{W(U)} 
                  = \ZZ_3[\ts t,p_1,p_2,c_5,p_3,p_4]
\label{eq:15}
\end{array}
\end{equation}
where  $c\ts_i=\sigma_i(t_1,t_2,\ldots ,t_5)$  and  
$p\tts_i=\sigma_i(t_1^2,t_2^2,\ldots,t_5^2)$  are
the elementary symmetric functions 
on \ts$t\tts_i$ \ts and \ts$t_j^2$, \ts respectively.

     We shall determine $H^*(BT\ts;\ZZ_3)^{W(E_6)}$,  that is, 
$R\tts_1$-invariant elements in $\ZZ_3[\ts t,p_1,p_2,c_5,p_3,p_4]$.

     We obtain, from the simple relation between  
$c\ts_i$'s and $p_j$'s, that
\begin{equation}
\begin{array}{lll}
      & c_2&=p_1-c_1^2,\cr
      & c_4&=-p_2+c_1^4+c_1^2p_1+c_1c_3+p_1^2,\cr
      & c_3c_5&=p_4-c_4^2,\cr
      & c_3^2 &=p_3+c_1^6+c_1^3c_3-c_1^2p_2-c_1c_3p_1
                                         +c_1c_5-p_1^3+p_1p_2.
\label{eq:16}
\end{array}
\end{equation}
The equality $0 =  \prod\limits_{i=1}^5 (t_1-t\tts_i)$  gives rise to
$$          t_1^5=c_5-t_1c_4+t_1^2c_3-t_1^3c_2+t_1^4c_1.$$
     Put  \ts$b=t_1+c_1$,  then we have
\begin{equation}
\begin{array}{ll}
&b^5=(-c_1^5-c_1^3p_1-c_1^2c_3+c_1p_1^2-c_1p_2+c_5)\cr
     & +\;b(-c_1^4-c_1^2p_1-p_1^2+p_2)
               +b^2(c_1^3+c_3)-b^3(c_1^2+p_1).
\label{eq:17}
\end{array}
\end{equation}
     From  (\ref{eq:14}),
and replacing 
$c_2$, $c_4$, $c_3c_5$, $c_3^2$  and  $b^5$  by  
(\ref{eq:16}) and  (\ref{eq:17}),  
we have
\begin{equation}
\begin{array}{lll}
 & R\tts_1(t)&=t-b,\cr
 & R\tts_1(b)&=-b,\cr
 & R\tts_1(c_1)&=c_1,\cr
 & R\tts_1(p_1)&=p_1, \cr
 & R\tts_1(p_2)&=p_2 ,\cr
 & R\tts_1(c_3)&=c_3-b(c_1^2+p_1)+b^3, \cr
 & R\tts_1(c_5)&=c_5+b(p_1^2-p_2)-b^3p_1,\cr
 & R\tts_1(p_3)&=p_3+b(-c_1^3p_1-c_1p_1^2+c_3p_1+c_5)
                                +b^2(c_1^2p_1+p_2)-b^3c_1p_1,\cr
 & R\tts_1(p_4)&=p_4+b(-c_1^5p_1+c_1^3p_1^2+c_1^3p_2-c_1^2c_3p_1
                                +c_3p_1^2-c_3p_2+c_5p_1)\cr
 &&\qquad \;+\;b^2(-c_1^4p_1-c_1^2p_2)
              +b^3(c_1^3p_1-c_1p_1^2+c_1p_2+c_3p_1+c_5)\cr
 &&\qquad \;+\;b^4(-c_1^2p_1-p_2).
\label{eq:18}
\end{array}
\end{equation}
Thus the following three elements are clearly  $R\tts_1$-invariant and hence $W(E_6)$-invariant:
\begin{equation}
 x_4=p_1, \qquad  x_8=p_2-p_1^2, \qquad y_{10}=c_5-tx_8-t^3x_4.
\label{eq:19}
\end{equation}
Put
\begin{equation}
        h_{12}=p_3+ty_{10}-t^2x_8-t^4x_4 \quad \mbox{and} \quad 
        h_{16}=p_4+t^3y_{10}-t^4x_8-t^6x_4.
\label{eq:110}
\end{equation}
Then we see that 
\begin{equation}
R\tts_1(h_{12})=h_{12}+d_8x_4 \quad\mbox{and}\quad 
    R\tts_1(h_{16})=h_{16}-d_8x_8,
\label{eq:111}
\end{equation}
where
\begin{equation}
d_8=b(-t^3-c_1^3-c_1x_4+c_3)+b^2(c_1^2+x_4)+b^3(t-c_1)-b^4
\label{eq:112}
\end{equation}
and
\begin{equation}
  R\tts_1(d_8)=-d_8. 
\label{eq:113}
\end{equation}

It follows that 
$$
 \ZZ_3[\ts t,p_1,p_2,c_5,p_3,p_4] 
                = \ZZ_3[\ts t,x_4,x_8,y_{10},h_{12},h_{16}].$$

     Further,  we put
\begin{equation}
h_{18}=t\ts(h_{16}-x_8^2)+t^3(-h_{12}+x_4x_8)+t^5x_8-t^7x_4+t^9
\label{eq:114}
\end{equation}
which is taken so as to satisfy
\begin{equation}
 R\tts_1(h_{18})=h_{18}+d_8y_{10}.
\label{eq:115}
\end{equation}
It follows from (\ref{eq:111}) and (\ref{eq:115})
that the following three elements are
$R\tts_1$-invariant and hence $W(E_6)$-invariant:
\begin{equation}
\begin{array}{ll}
& x_{20}=h_{12}x_8+h_{16}x_4,\cr 
& y_{22}=h_{12}y_{10}-h_{18}x_4,\cr
& y_{26}=h_{16}y_{10}+h_{18}x_8,
\label{eq:116}
\end{array}
\end{equation}
among which holds the following relation\ts:
\begin{equation}
 -x_{20}y_{10}+y_{22}x_8+y_{26}x_4=0.
\label{eq:117}
\end{equation}

Summing up we have found the first six $W(G)$-invariant elements\ts:
$$x_4,\  x_8,\  y_{10},\  x_{20},\  y_{22},\  y_{26}.$$

\section{The element  $x_{36}$.}

     In order to obtain another $R_1$-invariant element,  we consider
the following set  $S$  (cf. [TW]).  Put
\begin{equation}
     w_i  = 2\ts\tau_i -x  \qquad \mbox{for}\quad i = 1,2,\ldots ,6
\label{eq:21}
\end{equation}
and denote by  $S$  the set
$$ \{\;w_i +w_j \enskip\mbox{ for }\enskip i<j,
                      \quad x-w_i,\quad -x-w_i \;\}. $$

We see that the set  $S$  is invariant as a set under the action
of  $W(E_6)$.  Therefore the elementary symmetric functions  $\sigma_i^S$ %
on the 27 elements of \ts$S$
are invariant under the action of  $W(E_6)$.  

     We shall calculate 
\begin{equation}
\begin{array}{lll}
P =& 1+ \Sum_{j=1}^{27} \sigma_j^S 
      &= \Prod_{y\in S}(1+y)\cr
&     &= \Prod_{1\leq i<j\leq 6}(1+w_i +w_j)
                 \cdot \Prod_{1\leq j\leq 6}(1+x-w_j)\cdot \Prod_{1\leq j\leq 6} (1-x-w_j)
\label{eq:22}
\end{array}
\end{equation}
and decompose the result by degree to obtain $\sigma_j^S$.

\medskip
     Since we have
$$x={}-c_1, \quad w_1=t-c_1 \quad \mbox{and} \quad 
	w_i =t+c_1-t\ts_{i-1}  \enskip  \mbox{for}  \enskip i>1$$ %
by (\ref{eq:21}), (\ref{eq:11}), (\ref{eq:12})$'$  and  (\ref{eq:13}), 
the polynomial $P$ is expressed in terms of $c_i$'s and $t$.
(The calculation is carried out by Mathematica and the result has 2600 terms. )

We have from  (\ref{eq:16}), (\ref{eq:19}) and (\ref{eq:110})\; : 
\renewcommand{\arraystretch}{1.2}
\begin{equation}
\begin{array}{lrllllllll}
& c_2 & = & {}- c_1^2 + x_4,\\   
& c_4 & = & c_1^4 + c_1 c_3 + c_1^2 x_4 - x_8,\\
& c_5 & = & y_{10} + x_8 t + x_4 t^3,\\
& c_3^2 & = & h_{12} + t^4x_4 + t^2x_8 - t \ts y_{10} + x_4x_8
    + t^3x_4c_1 + t \ts x_8c_1 - x_4^2c_1^2 - x_4c_1c_3 
    - x_8c_1^2 \\
&&& {}+ y_{10}c_1 + c_1^6 + c_1^3c_3,\\
& c_3y_{10} & = & h_{16} + t^6x_4 + t^4x_8 - t^3y_{10} - x_8^2  
    - t^4x_4c_1^2 - t^3x_4c_1^3 - t^3x_4c_3 - t^2x_8c_1^2
    - t \ts x_8c_1^3 \\
& &&
    - t \ts x_8c_3 + t \ts y_{10}c_1^2 
    - h_{12}c_1^2 + x_4x_8c_1^2 + x_4c_1^6 - x_4c_1^3c_3 
    - x_8c_1c_3
    - y_{10}c_1^3 + c_1^8.
\label{eq:29}
\end{array}
\end{equation}

\medskip\noindent
Rewrite $P$ making use of (\ref{eq:29}) and 
we find no $c_3$ (as well as $c_2$, $c_4$ and $c_5$) in 
the result $P_1$ and $P_1 \in \ZZ_3[c_1,t,x_4,x_8,y_{10},h_{12},h_{16}]$.

\bigskip
We use two more relations to get rid of $c_1$'s. 

\bigskip
Replacing $c_3^2$ in the equality $c_3^2y_{10} - c_3\cdot c_3y_{10} 
= 0$, we obtain a relation:
\begin{eqnarray}
\label{eq:210}
& h_{16}c_3 = & 
   t^9x_4 + t^7x_4^2 + t^7x_8 - t^6y_{10} - t^5x_4x_8 
         + t^3h_{12}x_4 + t^3h_{16} + t^3x_4^2x_8 \\
  && +\; th_{12}x_8 + tx_4x_8^2 - ty_{10}^2 + h_{12}y_{10} 
         + x_4x_8y_{10} + t^7x_4c_1^2 + t^6x_4c_1^3 + t^6x_4c_3 \nonumber\\
  && +\; t^5x_4c_1^4 + t^5x_8c_1^2 - t^4x_4^2c_1^3 - t^4x_4x_8c_1 
         - t^4x_4c_1^5 - t^4x_4c_1^2c_3 + t^4x_8c_1^3 \nonumber\\
  && +\; t^4x_8c_3 - t^4y_{10}c_1^2 - t^3h_{12}c_1^2 - t^3x_4^3c_1^2 
         - t^3x_4^2c_1^4 + t^3x_4x_8c_1^2 - t^3x_4c_1^3c_3 \nonumber\\
  && +\; t^3x_8c_1^4 - t^3x_8c_1c_3 + t^3c_1^8 - t^2x_4x_8c_1^3 
        - t^2x_8^2c_1 - t^2x_8c_1^5 - t^2x_8c_1^2c_3 \nonumber\\
    && -\; t^2y_{10}c_1^4 + th_{12}c_1^4 - th_{16}c_1^2 - tx_4^2x_8c_1^2 
               + tx_4x_8c_1^4 + tx_4y_{10}c_1^3 - tx_4c_1^8 \nonumber\\
  && +\; tx_4c_1^5c_3 + tx_8^2c_1^2 + tx_8y_{10}c_1 - tx_8c_1^6 
          + tx_8c_1^3c_3 - tc_1^{10} - h_{12}x_4c_1^3 \nonumber\\
  && +\; h_{12}x_8c_1 + h_{12}c_1^5 + h_{12}c_1^2c_3 - h_{16}x_4c_1 
          - h_{16}c_1^3 - x_4^3c_1^5 - x_4^2x_8c_1^3 \nonumber\\
  && -\; x_4^2y_{10}c_1^2 - x_4^2c_1^7 - x_4x_8^2c_1
       + x_4x_8c_1^5 - x_4x_8c_1^2c_3 - x_4y_{10}c_1^4 - x_4c_1^9\nonumber\\
  && +\; x_4c_1^6c_3 + x_8^2c_3 + x_8c_1^7 - x_8c_1^4c_3 
          + y_{10}^2c_1 - y_{10}c_1^6 - c_1^{11} - c_1^8c_3,\nonumber
\end{eqnarray}

\bigskip
Rewriting $c_3^2$, $c_3y_{10}$ and $c_3h_{16}$ in  
the equality $c_3^2y_{10}^2 - (c_3y_{10})^2 = 0$  gives rise 
to the following relation.
\begin{eqnarray}
\label{eq:211}
 &c_1^{16}  =&
  -\;t^{12}x_4^2 + t^{10}x_4^3 + t^{10}x_4x_8 - t^9x_4y_{10}
              - t^8x_8^2 + t^7x_4^2y_{10} - t^7x_8y_{10}\\
&& +\;t^6h_{12}x_4^2 + t^6h_{16}x_4 + t^6x_4^3x_8 - t^6x_4x_8^2
              - t^6y_{10}^2 - t^5x_4x_8y_{10}\nonumber\\
&&
   -\;t^4h_{12}x_4x_8 + t^4h_{16}x_8 - t^4x_4^2x_8^2 - t^4x_4y_{10}^2
           - t^3h_{12}x_4y_{10} - t^3h_{16}y_{10}\nonumber\\
&&
   -\;t^3x_4^2x_8y_{10} - t^3x_8^2y_{10} + t^2h_{12}x_8^2 
   + t^2x_4x_8^3 - t^2x_8y_{10}^2 - th_{12}x_8y_{10}\nonumber\\
&& 
   -\;tx_4x_8^2y_{10} - ty_{10}^3 + h_{12}y_{10}^2 - h_{16}^2 
   - h_{16}x_8^2 + x_4x_8y_{10}^2 - x_8^4 - t^{10}x_4^2c_1^2\nonumber\\
&& 
   -\;t^8x_4^2c_1^4 + t^8x_4x_8c_1^2
              - t^7x_4y_{10}c_1^2 - t^6h_{12}x_4c_1^2 - t^6x_4^4c_1^2
              - t^6x_4^3c_1^4 \nonumber\\
&&
+\;t^6x_4^2x_8c_1^2 + t^6x_4^2c_1^6 - t^6x_4x_8c_1^4 + t^6x_4c_1^8 - t^6x_8^2c_1^2
              - t^5x_4y_{10}c_1^4 \nonumber\\
&& -\;t^5x_8y_{10}c_1^2
 + t^4h_{12}x_4c_1^4 - t^4h_{12}x_8c_1^2 - t^4h_{16}x_4c_1^2 + t^4x_4^3x_8c_1^2
             - t^4x_4^3c_1^6 \nonumber\\
&&+\; t^4x_4^2x_8c_1^4 
+ t^4x_4^2c_1^8 + t^4x_4x_8^2c_1^2 - t^4x_4x_8c_1^6 - t^4x_4c_1^{10}
              + t^4x_8c_1^8 - t^4y_{10}^2c_1^2 \nonumber\\
&&
-\; t^3h_{12}x_4x_8c_1 + t^3h_{12}y_{10}c_1^2 - t^3h_{16}x_4^2c_1 + t^3x_4^4c_1^5
              - t^3x_4^3x_8c_1^3 + t^3x_4^3y_{10}c_1^2 \nonumber\\
&&
+\; t^3x_4^3c_1 + t^3x_4^2x_8c_1^5 + t^3x_4^2y_{10}c_1^4 - t^3x_4x_8^2c_1^3
             - t^3x_4x_8y_{10}c_1^2 + t^3x_4x_8c_1^7 \nonumber\\
&&
+\; t^3x_4y_{10}^2c_1 - t^3x_4y_{10}c_1^6 + t^3x_8y_{10}c_1^4
              - t^3y_{10}c_1^8 + t^2h_{12}x_8c_1^4
             - t^2h_{16}x_8c_1^2 \nonumber\\
&&
-\; t^2x_4^2x_8^2c_1^2 - t^2x_4^2x_8c_1^6  - t^2x_4x_8^2c_1^4 + t^2x_4x_8c_1^8 + t^2x_8^2c_1^6
              - t^2x_8c_1^{10} \nonumber\\
&&-\; t^2y_{10}^2c_1^4 
- th_{12}x_8^2c_1 - th_{12}y_{10}c_1^4 - th_{16}x_4x_8c_1
              + th_{16}y_{10}c_1^2 + tx_4^3x_8c_1^5 \nonumber\\
&&-\; tx_4^2x_8^2c_1^3
              + tx_4^2x_8y_{10}c_1^2 + tx_4^2x_8c_1^7 + tx_4^2y_{10}c_1^6
              + tx_4x_8^2c_1^5 + tx_4x_8y_{10}c_1^4 \nonumber\\
&&-\; tx_4y_{10}c_1^8
              - tx_8^3c_1 + tx_8^2c_1^7 + tx_8y_{10}^2c_1
              - tx_8y_{10}c_1^6 + ty_{10}c_1^{10} - h_{12}^2c_1^4
              \nonumber\\
&&-\; h_{12}h_{16}c_1^2 
- h_{12}x_4^2c_1^6 - h_{12}x_4x_8c_1^4
              + h_{12}x_4c_1^8 - h_{12}x_8^2c_1^2 - h_{12}x_8y_{10}c_1\nonumber\\
&&             -\; h_{12}x_8c_1^6
              - h_{12}c_1^{10} 
           - h_{16}x_4^2c_1^4
              - h_{16}x_4y_{10}c_1 - h_{16}x_4c_1^6 + h_{16}x_8c_1^4
              + h_{16}c_1^8\nonumber\\
&&              -\; x_4^4c_1^8 + x_4^3x_8c_1^6
 + x_4^3y_{10}c_1^5
              - x_4^3c_1^{10} - x_4^2x_8y_{10}c_1^3 - x_4^2y_{10}^2c_1^2
              + x_4^2y_{10}c_1^7 \nonumber\\
&&              +\; x_4x_8^3c_1^2 + x_4x_8y_{10}c_1^5\
              - x_4c_1^{14} + x_8^3c_1^4 - x_8^2y_{10}c_1^3 
                         + x_8y_{10}^2c_1^2\nonumber\\
&&              +\; x_8y_{10}c_1^7 + x_8c_1^{12} 
        + y_{10}^3c_1 - y_{10}^2c_1^6.\nonumber
\end{eqnarray}

Rewrite $c_1^{16}$ in $P_1$ and the result $P_2$ has no $c_1$ and %
$P_2 \in  \ZZ_3[t,x_4,x_8,y_{10},h_{12},h_{16}]$.

The highest degree of \ $t$ \ is 27.
We use the relations obtained from (\ref{eq:114}), (\ref{eq:116})  
and  (\ref{eq:117}) to lower the degree of \;$t$ to 8\;:
\begin{eqnarray}
\label{eq:212}
&\hspace*{20pt} t^9 =& h_{18} + t^7x_4 - t^5x_8 + t^3h_{12} - t^3x_4x_8 
 	- t \ts h_{16} + t \ts x_8^2,\\
\label{eq:213}
& h_{18}x_8 =& {}- h_{16}y_{10} + y_{26}, \\
& h_{18}x_4 =& h_{12}y_{10} - y_{22}, \nonumber\\
& h_{16}x_4 =& {}- h_{12}x_8 + x_{20}, \nonumber\\
& y_{26}x_4 =& x_{20}y_{10} - y_{22}x_8\nonumber,\\
\label{eq:214}
& h_{18}x_{20} =& h_{12}y_{26} - h_{16}y_{22}.
\end{eqnarray}

\bigskip
     The result $P_3$ is as follows. 
\begin{equation}
\begin{array}{llll}
\label{eq:215}
&  P_3 = &
1 +h_{12} x_4 y_{26} -h_{12} x_4^2 x_8^2 +h_{12} x_4^2 y_{10} t^3 +h_{12} x_8 y_{22} +h_{12} x_{20} x_4 -h_{12} x_{20} y_{10} 
+h_{12}^2 x_4 x_8 
\\&&{}
-h_{12}^3 
-h_{16} h_{18} y_{10}^2 +h_{16} x_8 y_{10}^3 -h_{16} x_8^2 y_{10} t^3 +h_{16} x_8^4 +h_{16}^2 x_8^2 
+h_{16}^3 -h_{18} y_{10} y_{26} 
\\&&{}
+h_{18} y_{10}^3 t^3 -h_{18}^3
 -x_4 x_8 +x_4 x_8 y_{10}^3 +x_4 x_8 y_{10}^3 t^6 
+x_4 x_8 y_{10}^4 -x_4 x_8 y_{10}^4 t +x_4 x_8^2 y_{10} 
\\&&{}
-x_4 x_8^2 y_{10}^2 +x_4 x_8^2 y_{26} t +x_4 x_8^3 y_{10}^2 +x_4 x_8^4 
+x_4 x_8^4 t^6 -x_4 x_8^4 y_{10} t -x_4 y_{10} y_{22} +x_4 y_{10}^2 
\\&&{}
+x_4 y_{10}^3 -x_4 y_{26} -x_4^2 x_8 
+x_4^2 x_8 y_{10}^2 -x_4^2 x_8 y_{10}^3 +x_4^2 x_8 y_{10}^3 t^4 +x_4^2 x_8 y_{26} -x_4^2 x_8^2 y_{10} 
\\&&{}
+x_4^2 x_8^3 y_{10} t^3 +x_4^2 x_8^4 
+x_4^2 x_8^4 t^4 -x_4^2 y_{10} -x_4^2 y_{10}^2 -x_4^2 y_{10}^4 +x_4^2 y_{10}^4 t^3 +x_4^2 y_{22} -x_4^2 y_{22} t^3 
\\&&{}
+x_4^2 y_{26} -x_4^3 -x_4^3 x_8 t^8 -x_4^3 x_8 y_{10} t^3 -x_4^3 x_8^2 -x_4^3 x_8^2 t^4 -x_4^3 x_8^3 
-x_4^3 x_8^3 t^6 +x_4^3 y_{10} t^7 
\\&&{}
+x_4^3 y_{10}^2 -x_4^3 y_{10}^2 t^2 +x_4^3 y_{10}^3 -x_4^3 y_{10}^3 t^6 +x_4^3 y_{22} 
-x_4^3 y_{22} t -x_4^4 x_8 +x_4^4 x_8 t^6 +x_4^4 x_8 y_{10} 
\\&&{}
-x_4^4 x_8 y_{10} t -x_4^5 x_8 +x_4^5 x_8 t^4 
-x_4^5 y_{10} +x_4^5 y_{10} t^3 +x_4^6 -x_4^6 t^6 -x_8 y_{10} +x_8 y_{10}^2 
\\&&{}
-x_8 y_{10}^2 y_{26} 
-x_8 y_{10}^3 t^8 +x_8 y_{10}^4 -x_8 y_{10}^4 t^3 -x_8 y_{26} -x_8^2 -x_8^2 y_{10} y_{26} -x_8^2 y_{10}^2 
-x_8^2 y_{10}^3 
\\&&{}
-x_8^2 y_{10}^3 t^4 +x_8^2 y_{26} +x_8^2 y_{26} t^3 -x_8^3 +x_8^3 y_{10} +x_8^3 y_{10} t^7 
+x_8^3 y_{10}^2 -x_8^3 y_{10}^2 t^2 -x_8^4 t^8 
\\&&{}
-x_8^4 y_{10} t^3 -x_8^4 y_{10}^2 +x_8^5 -x_8^5 t^4 
-x_{20} x_4 +x_{20} x_4 x_8^2 +x_{20} x_4^2 -x_{20} x_4^2 x_8 -x_{20} x_4^3 
\\&&{}
+x_{20} x_4^3 t^2 -x_{20} x_8 
-x_{20} x_8 y_{10}^2 -x_{20} x_8^2 -x_{20} x_8^2 y_{10} t -x_{20} x_8^3 +x_{20} x_8^3 t^2 -x_{20} y_{10} 
\\&&{}
-x_{20} y_{10}^2 
-x_{20} y_{10}^3 +x_{20} y_{10}^3 t^2 -x_{20} y_{22} +x_{20} y_{26} -x_{20}^2 -y_{10} y_{26} +y_{10}^2 y_{22} 
-y_{10}^2 y_{26} 
\\&&{}
-y_{10}^3 +y_{10}^3 y_{22} -y_{10}^3 y_{22} t -y_{10}^4 +y_{10}^4 t^7 +y_{10}^5 
-y_{10}^5 t^2 +y_{22} y_{26} -y_{26}^2 
\end{array}
\end{equation}

Now we decompose $P_3$ by degree to obtain $\sigma_j^S$. 

The $\sigma_j^S$ for  $j \leq 17$  are as follows\ts:
\begin{equation}
\begin{array}{lll}
\label{eq:214}
& \sigma_i^S&= 0 \hskip40mm\mbox{for }\  1\leq i\leq 5  
                \   \mbox{ and }\  i=7,10,11,13,\cr
      & \sigma_6^S&= -x_4^3-x_4x_8,\cr
      & \sigma_8^S&= -x_4^2x_8-x_8^2,\cr
      & \sigma_9^S&= -x_4^2y_{10}-x_8y_{10},\cr
      & \sigma_{12}^S& = -x_{20}x_4+x_4^6-x_4^4x_8
                            +x_4y_{10}^2-x_8^3,\cr
      & \sigma_{14}^S& = x_{20}(x_4^2-x_8)-x_4^5x_8-x_4^3x_8^2
                         -x_4^2y_{10}^2+x_8y_{10}^2,\cr
      & \sigma_{15}^S& = x_{20}y_{10}+y_{22}(x_4^2+x_8)-x_4^5y_{10}
                       +x_4x_8^2y_{10}-y_{10}^3,\cr
      & \sigma_{16}^S &= -x_{20}x_4^3+x_4^3y_{10}^2,\cr
      & \sigma_{17}^S& = x_{20}x_4y_{10}+y_{22}(x_4^3-x_4x_8)-y_{26}x_8\cr
      &&\hskip40mm  +\;x_4^4x_8y_{10}-x_4^2x_8^2y_{10}+x_4y_{10}^3+x_8^3y_{10}.\cr
\end{array}
\end{equation}

     The element  $\sigma_{18}^S$  is expressed as
$$\begin{array}{llllllll}
\sigma_{18}^S =\; 
& -\; x_{20}x_4^2x_8-x_{20}x_8^2-y_{22}x_4y_{10}-y_{26}y_{10}
             -x_4^3x_8^3+x_4^2x_8y_{10}^2+x_4x_8^4\\
& -\;x_8^2y_{10}^2-h_{12}^3+h_{12}^2x_4x_8+h_{12}x_{20}x_4
             -h_{12}x_4^2x_8^2-t^8x_4^3x_8+t^7x_4^3y_{10}\\
& -\;t^6x_4^6+t^6x_4^4x_8+t^4x_4^5x_8-t^4x_4^3x_8^2
             +t^3h_{12}x_4^2y_{10}-t^3y_{22}x_4^2\\
& +\;t^3x_4^5y_{10}-t^3x_4^3x_8y_{10}+t^2x_{20}x_4^3-t^2x_4^3y_{10}^2
             -ty_{22}x_4^3-tx_4^4x_8y_{10}
\end{array}$$
which cannot be expressed in terms of the $R_1$-invariant elements 
already obtained. We put
\begin{eqnarray}
\label{eq:215}
&g_{24} =\;& -\;t^8x_8+t^7y_{10}-t^6x_4^3+t^6x_4x_8
                    +t^4x_4^2x_8-t^4x_8^2+t^3h_{18}\\
&& +\;t^3x_4^2y_{10}-t^3x_8y_{10}+t^2x_{20}
                    -t^2y_{10}^2-ty_{22}-tx_4x_8y_{10}\nonumber
\end{eqnarray}
and
\begin{equation}
\begin{array}{lll}
\label{eq:216}
&x_{36} =\;& -\;g_{24}x_4^3+h_{12}^3-h_{12}^2x_4x_8
                   -h_{12}x_{20}x_4+h_{12}x_4^2x_8^2.
\end{array}
\end{equation}
Then
$$\begin{array}{llllllll}
R_1(h_{12}^3-h_{12}^2x_4x_8-h_{12}x_{20}x_4+h_{12}x_4^2x_8^2)\cr
\mbox{}\hskip20mm  =\; (h_{12}^3-h_{12}^2x_4x_8
           -h_{12}x_{20}x_4+h_{12}x_4^2x_8^2)\cr
\mbox{}\hskip25mm
 +\;x_4^3(d_8^{\ts3}-d_8^{\ts2}x_8-d_8h_{16}+d_8x_8^2)
\end{array}$$
and
\begin{equation}
 R_1(g_{24}) 
            = g_{24}+d_8^{\ts3}-d_8^{\ts2}x_8-d_8h_{16}+d_8x_8^2
\label{eq:217}
\end{equation}
hold. The element  $x_{36}$ is shown to be $R_1$-invariant and we have
\begin{equation}
\begin{array}{lll}
\label{eq:218}
& \sigma_{18}^S = \;
         &-\; x_{36}+x_{20}(-x_4^2x_8-x_8^2)
                        -y_{22}x_4y_{10}-y_{26}y_{10}\cr
       & &\hskip30mm -\;x_4^3x_8^3+x_4^2x_8y_{10}^2
                        +x_4x_8^4-x_8^2y_{10}^2.\cr
\end{array}
\end{equation}
     For the sake of completeness, we list the result of $ \sigma_j^S$ for  $j>18$,
where $x_{48}$ and $x_{54}$ are the elements found in [MS] and will be studied 
in \S 4.

\bigskip
$\begin{array}{llllllll}
\hspace{1.2em}  
& \sigma_{20}^S & = & 
{}-x_{20}^2 +x_{20} x_4 x_8^2 +x_4^2 x_8^4 +x_8^5
-x_{20} y_{10}^2 -x_4 x_8^2 y_{10}^2 -y_{10}^4,\\
& \sigma_{21}^S & = & 
x_{20} x_4 x_8 y_{10} + x_4^3 y_{10}^3 +x_4 x_8 y_{10}^3
-x_{20} y_{22} -x_4 x_8^2 y_{22} +y_{10}^2 y_{22} +x_8^2 y_{26},\\
& \sigma_{22}^S & = & 
{}-x_{20} x_8^3 +x_8^3 y_{10}^2,\\
& \sigma_{23}^S & = & 
{}-x_4^2 x_8 y_{10}^3 -x_8^2 y_{10}^3 +x_{20} y_{26} -y_{10}^2 y_{26},\\
& \sigma_{24}^S & = & 
x_{48} -x_{20} x_8 y_{10}^2
+x_4 x_8^3 y_{10}^2 - x_4^2 y_{10}^4 + x_8 y_{10}^4 + y_{22} y_{26},\\
& \sigma_{25}^S & = & 
{}-x_{20} y_{10}^3 +y_{10}^5,\\
& \sigma_{26}^S & = & 
{}-x_8^4 y_{10}^2 + x_4x_8y_{10}^4 + y_{10}^3y_{22} 
- x_8^2 y_{10}y_{26} -y_{26}^2,\\
& \sigma_{27}^S & = & 
{}-x_{54}.
\end{array}$

\section{The invariant subalgebra  $H^*(BT;\ZZ_3)^{W(E_6)}$}

{\bf Notation} \quad 
Let $\Omega$ be a field extension over a field $k$. 
For elements $x_1$, \ldots , $x_n$ of $\Omega$, 
we denote an algebra generated by  $x_1$, \ldots , $x_n$ over $k$
by $k[x_1, \ldots , x_n]$ and its quotient field 
by $k(x_1, \ldots , x_n)$.
In this section, we note that we use 
the notations $k[x_1, \ldots , x_n]$ and $k(x_1, \ldots , x_n)$ by 
the above extended meaning.

For a field $k$ and  $\Omega$, we take $\ZZ_3$ and 
\ $\ZZ_3(t,\, t_1,\,  \ldots , t_6)$ respectively.
We put 
\ $K = \ZZ_3(x_4, x_8, y_{10}, x_{20}, y_{22}, x_{36})$ 
\ and
\ $L = \ZZ_3(t, x_4, x_8, y_{10}, h_{12}, h_{16})$.
We see $K \subset L$.

\begin{Lemma}\label{31}
{\sl We have} \ $L=K(t)$. 
{\sl The extension degree  $[L:K]$ is at most}  27.
\end{Lemma}

\begin{proof}
From the definition, we see $K \subset L$ and $t \in L$.
It means $K(t)\subset L$, where $K(t)$ is 
a field generated by $K$ and $t$ in $L$.
To show $L \subset K(t)$, it is sufficient to show 
that $h_{12} \in K$ Because it implies that 
\ $h_{16}=(x_{20} - h_{12}x_8)/x_4 \in K(t)$.

Replacing $h_{18}$ by its defining expression (\ref{eq:114})
in $y_{22}=h_{12}y_{10} - h_{18}x_4$,   we obtain
\begin{equation}
y_{22} = (y_{10} + tx_8 + t^3x_4)h_{12} - t(x_{20} 
                    - x_4x_8^2) -  t^3x_4^2x_8  -  t^5x_4x_8  
                    +  t^7x_4^2  -  t^9x_4.
\label{eq:32}
\end{equation}
Hence
\begin{equation}
 (y_{10} + tx_8 + t^3 x_4)h_{12} 
    = y_{22} + t(x_{20} - x_4 x_8^2) + t^3 x_4^2 x_8 + t^5 x_4x_8 
      - t^7x_4^2 + t^9x_4.
\label{eq:33}
\end{equation}
Thus  $h_{12} \in K(t)$  and so does $h_{16}$.
We have shown that $L=K(t)$.

\bigskip
Now we shall find a equation of \,$t$ over $K$.
$$0 = x_{36} + g_{24}x_4^3 
 -  h_{12}^3 + h_{12}^2x_4x_8 + h_{12}x_{20}x_4 - h_{12}x_4^2x_8^2.
\leqno(3.16)'$$
Replacing  $g_{24}$  by its defining expression (\ref{eq:215}),
$h_{16}x_4$  by $x_{20} - h_{12}x_8$  and %
$h_{18}x_4$  by  $ - y_{22} + h_{12}y_{10}$,   
we obtain a polynomial in
$t$, $x_4$, $x_8$, $y_{10}$, $x_{20}$, $y_{22}$, $x_{36}$ and $h_{12}$.
Since its degree with respect to  $h_{12}$  is  3,   
we multiply (\ref{eq:216})%
$'$ by $(y_{10} + tx_8 + t^3x_4)^3$.  Then
by making use of (\ref{eq:33})
we obtain a polynomial in  
$\ZZ_3[\ts t, x_4, x_8, y_{10}, x_{20}, y_{22}, x_{36}]$  of degree  27  
with respect to  $t$\,:
\begin{equation}
\begin{array}{lll}
	\label{eq:35}
& 0=&x_{36}y_{10}^3
+x_{20}y_{22}x_4y_{10}^2
+y_{22}^2x_4x_8y_{10}
-y_{22}^3
-y_{22}x_4^2x_8^2y_{10}^2
+tx_4^3x_8^4y_{10}^2
\\ & & {}
-tx_4^4x_8y_{10}^4
+tx_{20}^2x_4y_{10}^2
+tx_{20}x_4^2x_8^2y_{10}^2
+tx_{20}y_{22}x_4x_8y_{10}
+ty_{22}^2x_4x_8^2
\\ & & {}
-ty_{22}x_4^2x_8^3y_{10}
-ty_{22}x_4^3y_{10}^3
-t^2x_4^3y_{10}^5
+t^2x_{20}x_4^3y_{10}^3
-t^3x_4^3x_8y_{10}^4
\\ & & {}
-t^3x_4^4x_8^3y_{10}^2
+t^3x_4^5y_{10}^4
-t^3x_{20}^2x_4x_8^2
-t^3x_{20}^3
-t^3x_{20}x_4^2x_8^4
+t^3x_{20}x_4^3x_8y_{10}^2
\\ & & {}
-t^3x_{20}y_{22}x_4^2y_{10}
+t^3x_{36}x_8^3
+t^3y_{22}^2x_4^2x_8
+t^4x_4^3x_8^2y_{10}^3
+t^4x_4^5x_8y_{10}^3
\\ & & {}
-t^4x_{20}^2x_4^2y_{10}
+t^4x_{20}x_4^2y_{10}^3
+t^4x_{20}y_{22}x_4^2x_8
-t^4y_{22}x_4^2x_8y_{10}^2
-t^5x_4^3x_8^3y_{10}^2
\\ & & {}
+t^5x_{20}x_4^3x_8^3
-t^6x_4^4x_8y_{10}^3
-t^6x_4^6y_{10}^3
-t^6x_{20}x_4^2x_8^2y_{10}
-t^6x_{20}x_4^4x_8y_{10}
\\ & & {}
+t^6x_{20}y_{22}x_4^3
+t^6y_{22}x_4^2x_8^3
-t^6y_{22}x_4^3y_{10}^2
+t^6y_{22}x_4^4x_8^2
-t^7x_4^3x_8^5
+t^7x_4^3y_{10}^4
\\ & & {}
+t^7x_4^4x_8^2y_{10}^2
-t^7x_4^5x_8^4
+t^7x_{20}^2x_4^3
+t^7x_{20}x_4^3y_{10}^2
-t^7x_{20}x_4^4x_8^2
+t^9x_4^3x_8^2y_{10}^2
\\ & & {}
-t^9x_4^4x_8^4
-t^9x_4^5x_8y_{10}^2
+t^9x_4^6x_8^3
+t^9x_{20}x_4^2y_{10}^2
+t^9x_{20}x_4^3x_8^2
\\ & & {}
+t^9x_{20}x_4^5x_8
+t^9x_{36}x_4^3
-t^9y_{22}x_4^2x_8y_{10}
+t^9y_{22}x_4^4y_{10}
-t^{10}x_4^3x_8^3y_{10}
\\ & & {}
-t^{10}x_4^4y_{10}^3
+t^{10}x_4^5x_8^2y_{10}
-t^{10}x_4^7x_8y_{10}
+t^{10}x_{20}x_4^2x_8y_{10}
-t^{10}x_{20}x_4^4y_{10}
\\ & & {}
-t^{10}y_{22}x_4^2x_8^2
+t^{10}y_{22}x_4^4x_8
-t^{10}y_{22}x_4^6
-t^{11}x_4^6y_{10}^2
+t^{11}x_{20}x_4^6
+t^{12}x_4^3y_{10}^3
\\ & & {}
-t^{12}x_4^4x_8^2y_{10}
+t^{12}x_4^8y_{10}
-t^{12}x_{20}x_4^3y_{10}
-t^{12}y_{22}x_4^3x_8
-t^{12}y_{22}x_4^5
\\ & & {}
-t^{13}x_4^3x_8y_{10}^2
+t^{13}x_4^5y_{10}^2
+t^{13}x_4^6x_8^2
+t^{13}x_4^8x_8
+t^{13}x_{20}x_4^3x_8
-t^{13}x_{20}x_4^5
\\ & & {}
+t^{15}x_4^3x_8^3
-t^{15}x_4^4y_{10}^2
+t^{15}x_4^7x_8
-t^{15}x_4^9
+t^{15}x_{20}x_4^4
+t^{18}x_4^3x_8y_{10}
\\ & & {}
+t^{18}x_4^5y_{10}
+t^{19}x_4^3x_8^2
+t^{19}x_4^5x_8
+t^{21}x_4^4x_8
+t^{21}x_4^6
-t^{27}x_4^3.
\end{array}
\end{equation}

Hence we conclude $[L:K]=[K(t):K]\le 27$.
\end{proof}
\begin{Corollary}\label{32}
{\sl The elements $x_4$, $x_8$, $y_{10}$, $x_{20}$, $y_{22}$, $x_{36}$ are
algebraically independent over $\ZZ_3$.}
\end{Corollary}
\begin{proof}
Since $L$ is an algebraic extension of $K$ and 
the transcendental degree $\trdeg_{\ZZ_3} L$ of $L$ is 6, 
we have $\trdeg_{\ZZ_3} K=6$. Furthermore $K$ is generated 
by the 6 elements $x_4$, $x_8$, $y_{10}$, $x_{20}$, $y_{22}$, $x_{36}$ 
\ over $\ZZ_3$. So we get Corollary \ref{32}.
\end{proof}

\begin{Lemma}\label{33}
$\Big[\, L : \ZZ_3(t, t_1, t_2, t_3, t_4, t_5)^{W(E_6)} \,\Big]=27$.
\end{Lemma}

\begin{proof}
We have the inclusions of fields

\medskip
\mbox{}\hskip6em $
\ZZ_3(t, t_1, t_2, t_3, t_4, t_5)^{W(E_6)}
\subset \ZZ_3(t, t_1, t_2, t_3, t_4, t_5)^{W(Spin(10))}\\
\mbox{}\hskip10em 
= \ZZ_3(t, x_4, x_8, y_{10}, h_{12}, h_{16}) = L 
\subset \ZZ_3(t, t_1, t_2, t_3, t_4, t_5)
$

In order to apply the Galois theory to our case, we need to show 
that $W(E_6)$ acts on $H^*(BT;\ZZ_3)$ as automorphisms, 
where $T$ is the maximal torus of $E_6$ and the quotient field 
of $H^*(BT;\ZZ_3)$ is $\ZZ_3(t, t_1, t_2, t_3, t_4, t_5)$.

When we define $ H=\{g \in W(E_6) 
                        \mid gx=x, \ \forall x \in H^*(BT;\ZZ_3) \} $,
$H$ is a normal subgroup of $W(E_6)$.
According to \cite{C}, the Weyl group $W(E_6)$ contains 
the simple group $SU_4(2)$ with index 2.
Hence $H$ is $SU_4(2)$ or $\{e\}$.
Suppose $H=SU_4(2)$, then $W(E_6)$ acts on $H^*(BT;\ZZ_3)$ as 
an automorphism with the order 2. This is a contradiction from 
the argument of \S1.
Thus we can apply the Galois theory in our context.

\medskip\noindent
{\bf Remark.}
This fact, however, can be checked by appealing to Lemma 10.7.1
of \cite{LS}.
Thus we can apply Galois theory in our context
(thanks to Proposition 1.2.4 of \cite{LS}).

     As is well known  \cite{F},  the order of  $W(E_6)$ 
is  $2^7\cdot 3^4 \cdot 5$  \medskip and
that of  $W(Spin(10))$ is  $2^7\cdot 3 \cdot 5$.   
Noting that $L = K(t, t_1, t_2, t_3, t_4, t_5)^{W(Spin(10))}$,
we have
$$\Big[L : \ZZ_3(t, t_1, t_2, t_3, t_4, t_5)^{W(E_6)} \Big]
        = \Frac{\vphantom{\Big(} | \,W(E_6) \,|}
               {\vphantom{\Big(} | \,W(Spin(10)) \,|} 
        = 3^3.$$
\end{proof}

\begin{Lemma}\label{34}
$K = \ZZ_3(t, t_1, t_2, t_3, t_4, t_5)^{W(E_6)} $.
\end{Lemma}

\begin{proof}
From Lemma \ref{31},
we have
$$ K \subset \ZZ_3(t, t_1, t_2, t_3, t_4, t_5)^{W(E_6)} 
\subset L = K(t). $$
Hence 
$\Big[\, L : \ZZ_3(t, t_1, t_2, t_3, t_4, t_5)^{W(E_6)} \,\Big] $ %
divides $[L : K]$. Further we have 
\[
[L : K] 
\le \Big[\, L : \ZZ_3(t, t_1, t_2, t_3, t_4, t_5)^{W(E_6)}\, \Big]
= 27\]
\ from Lemmas \ref{31} and \ref{33}.
It follows that 
\ $K = \ZZ_3(t, t_1, t_2, t_3, t_4, t_5)^{W(E_6)} $.
\end{proof}

To state the next lemma, we use the following notation:\\
Let $A$ be a ring and $S$ be a set.
Then denote by $A\{s\}$ a free $A$-module $
\mathop{\oplus}\limits_{s \in S}A\{s\}$
with a basis $S$. we set
$$\begin{array}{ll}
  H        &= \ZZ_3[\ts t, x_4, x_8, y_{10}, h_{12}, h_{16}]\ts, \\
  M        &= \ZZ_3[x_4, x_8, y_{10}, x_{20}, y_{22}, x_{36}]\ts,\\
 \wtilde M &= \ZZ_3[x_4,x_4^{-1}][x_8, y_{10}, x_{20}, y_{22}, x_{36}]\ts.
\end{array}$$

\begin{Lemma}\label{35}
$(1)$ \enskip
$L=K\{\,t_{}^ih_{12}^j \mid \ 0 \le i \le 8, \ 0 \le j \le 2 \}$.
In particular $t_{}^ih_{12}^j$, $0 \le i \le 8$, $0 \le j \le 2$ is
linearly independent over $M$ and $\wtilde M$.

$(2)$ \enskip $H \subset 
\wtilde M \{\,t_{}^ih_{12}^j \mid \ 0 \le i \le 8, \ 0 \le j \le 2 \}$.
\end{Lemma}

\begin{proof}
(1) \enskip Since $K(t)$ is an algebraic extension 
of $K$ and $h_{12} \in L=K(t)$, we see 
$$L = K(t) = K[t] = K[h_{12},t].$$

From the relation 
$$ t^9x_4 = t^7x_4^2 - t^5x_4x_8 + t^3h_{12}x_4 - t^3x_4^2x_8
            + th_{12}x_8 - tx_{20}+ tx_4x_8^2 + h_{12}y_{10} - y_{22},
\leqno(\ref{eq:33})'$$
we obtain $L = K[h_{12}]\{\, t^i \mid 0 \le i \le 8 \}$. At the same time,
the formula (\ref{eq:33})$'$
asserts that $t$ is 
integral over $\wtilde M[h_{12}]$.

From (\ref{eq:215}) and (\ref{eq:216}),
we have $h_{12}^3 \in 
K\{\, t_{}^ih_{12}^j \mid \ 0 \le i \le 8, \ 0 \le j \le 2 \}$ and
hence $L = K[h_{12}]\{\, t^i \mid 0 \le i \le 8 \} 
= K\{\, t_{}^ih_{12}^j \mid \ 0 \le i \le 8, \ 0 \le j \le 2 \}$.

On the other hand, this means that $\dim_K L \le 27$.

From Lemmas \ref{33} and \ref{34},
$\{\, t_{}^ih_{12}^j \mid \ 0 \le i \le 8, \ 0 \le j \le 2 \}$ is linear 
independent over $K$. 
Noting that quotient field of  $M$ and $\wtilde M$ is $K$, 
the last statement is clear.

(2) \enskip First we show that 
\ $\ZZ_3[\ts t, x_4, x_8, y_{10}, h_{12}] \in 
\wtilde M\{\, t_{}^ih_{12}^j \mid \ 0 \le i \le 8, \ 0 \le j \le 2 \}$.
To prove this assertion, it is enough to show that 
\ $t^n, h_{12}^n \in 
\wtilde M\{\, t_{}^ih_{12}^j \mid \ 0 \le i \le 8, \ 0 \le j \le 2 \}$.

As is shown in the above proof, we have
$t^n%
\in 
M[h_{12}]\{\, t^i \mid 0 \le i \le 8 \}$.
Next we will~ show that 
$\wtilde M[h_{12}]\{\, t^i \mid 0 \le i \le 8 \} 
= \wtilde M\{\, t_{}^ih_{12}^j \mid \ 0 \le i \le 8, \ 0 \le j \le 2 \}$.
We substitute $h_{16}=(x_{20} - h_{12}x_8)/x_4$ for $h_{16}$ in 
(\ref{eq:114})
and obtain (recalling 
$h_{18}=t\ts(h_{16}-x_8^2)+t^3(-h_{12}+x_4x_8)+t^5x_8-t^7x_4+t^9$)
$$ h_{18} = t\ts(x_{20}-h_{12}x_8)/x_4 -tx_8^2 +t^3(-h_{12}+x_4x_8)
                +t^5x_8 -t^7x_4 +t^9.
\leqno(\ref{eq:114})'
$$
We replace $h_{18}$ in (\ref{eq:215})
 by the righthand side of (\ref{eq:114})$'$
and obtain (recalling 
\[
\begin{array}{lll} g_{24} &=& -t^8x_8 +t^7y_{10} -t^6x_4^3 +t^6x_4x_8 +t^4x_4^2x_8 -t^4x_8^2
         +t^3h_{18} +t^3x_4^2y_{10} \\ 
   &&      -t^3x_8y_{10}+t^2x_{20}-t^2y_{10}^2-ty_{22}-tx_4x_8y_{10})\end{array}
\]

\noindent
$(\ref{eq:215})'
\qquad g_{24} = t^{12} -t^{10}x_4 +t^7x_4x_8+t^7y_{10}-t^7h_{12}
                -t^6x_4^3+t^6x_4x_8+t^4x_4^2x_8+t^4x_8^2
            \\ \mbox{}\hskip8em {}
                +t^4x_{20}/x_4 - t^4h_{12}x_8/x_4 
                +t^3x_4^2y_{10}-t^3x_8y_{10}+t^2x_{20}
                -t^2y_{10}^2-ty_{22}-tx_4x_8y_{10}.$

\bigskip
We replace $g_{24}$ in (\ref{eq:216})
by the righthand side of (\ref{eq:115})$'$
 and obtain (recalling 
$x_{36} = -g_{24}x_4^3+h_{12}^3-h_{12}^2x_4x_8
                   -h_{12}x_{20}x_4+h_{12}x_4^2x_8^2$)

\noindent
\begin{equation}
\begin{array}{l}
\label{eq:36}
 h_{12}^3 = - t^8x_4^3x_8 + t^7x_4^3y_{10} - t^6x_4^6
            + t^6x_4^4x_8 + t^4x_4^5x_8 - t^4x_4^3x_8^2
            + t^3 h_{12}x_4^2 y_{10} - t^3y_{22}x_4^2 
            \\ 
            + t^3x_4^5y_{10}
            - t^3x_4^3x_8y_{10} + t^2x_{20}x_4^3 - t^2x_4^3y_{10}^2
            -ty_{22}x_4^3 - tx_4^4x_8y_{10} 
            \\ 
            + h_{12}^2x_4x_8
            + h_{12}x_{20}x_4 - h_{12}x_4^2x_8^2 + x_{36}.
\end{array}
\end{equation}

From (\ref{eq:33})$'$ and (\ref{eq:36}),
when we note that $t^n\in 
\wtilde M\{\, t_{}^i, h_{12}t_{}^i \mid \ 0 \le i \le 8 \}$ for $ n=10,12 $,
we have 
\ $h_{12}^n \in 
\wtilde M\{\, t_{}^ih_{12}^j \mid \ 0 \le i \le 8, \ 0 \le j \le 2 \}$.
$$\wtilde M[h_{12}]\{\, t_{}^i \mid \ 0 \le i \le 8 \}
= \wtilde M\{\, t_{}^ih_{12}^j \mid \ 0 \le i \le 8, \ 0 \le j \le 2 \}.$$
\end{proof}

\begin{Theorem}\label{36}
$H^*(BT;\ZZ_3)^{W(E_6)} = H \cap \wtilde M$. 
\end{Theorem}
\begin{proof}
From Lemma \ref{34}, 
we have $H^*(BT;\ZZ_3)^{W(E_6)} = H \cap K$. 

Using Lemma \ref{35} 
(2), we see
$$H \cap K \subset
 \wtilde M\{\, t_{}^ih_{12}^j \mid \ 0 \le i \le 8, \ 0 \le j \le 2 \}
 \cap K.$$
We note that $K$ is the quotient field of $\wtilde M$ and
that $\{\, t_{}^ih_{12}^j \mid \ 0 \le i \le 8, \ 0 \le j \le 2 \}$ is
linear independent over $\wtilde M$ and $K$ from Lemma \ref{35} 
(1).
Hence we get
$$ \wtilde M\{\, t_{}^ih_{12}^j \mid \ 0 \le i \le 8, \ 0 \le j \le 2 \}
 \cap K = \wtilde M.$$
This implies that $H^*(BT;\ZZ_3)^{W(E_6)} \subset H \cap \wtilde M$. 

On the hand, we see that 
\ $H \cap \wtilde M \subset H \cap K = H^*(BT;\ZZ_3)^{W(E_6)}$.
\ So we have proved the theorem. 
\end{proof}

We state the main theorem in this section, that is,
$H^*(BT;\ZZ_3)^{W(E_6)}$ is generated 
by the thirteen elements\ts:
$$x_4,\  x_8,\  y_{10},\  x_{20},\  y_{22},\  y_{26},\  
x_{36},\  x_{48},\  x_{54},\  y_{58},\  y_{60},\  y_{64},\  y_{76}.$$
These elements except $x_4,\  x_8,\  y_{10},\  x_{36}$ are defined 
in $H \cap \wtilde M$ as follows:
\begin{eqnarray}
        &x_{20} &= h_{12}x_8+h_{16}x_4,\nonumber \\ 
        &y_{22} &= h_{12}y_{10}-h_{18}x_4,\nonumber\\
        &y_{26} &= (x_{20}y_{10}-y_{22}x_8)/x_4 = h_{16}y_{10}+h_{18}x_8,\nonumber\\
\label{eq:37}
   &x_{48} &= (-x_{36}x_8^3 + x_{20}^3 + x_{20}^2x_4x_8^2
                   + x_{20}x_4^2x_8^4)/x_4^3 \\
        &       &= h_{16}^3 + h_{16}^2 x_8^2 + h_{16}x_8^4 + g_{24}x_8^3,\nonumber\\
\label{eq:38}
 &x_{54} &= (x_{36}y_{10}^3 - y_{22}^3 + x_{20}y_{22}x_4y_{10}^2
                   + y_{22}^2x_4x_8y_{10} - y_{22}x_4^2x_8^2y_{10}^2)
                                                               /x_4^3 \\
        &       &= h_{18}^3 - h_{16}h_{18}y_{10}^2 + h_{18}^2x_8y_{10}
                   + h_{18}x_8^2y_{10}^2 - g_{24}y_{10}^3,\nonumber\\
\label{eq:39}
   &y_{58} &= (x_{36}x_8^2y_{10} - x_{20}^2y_{22})/x_4 \\
        &       &= (h_{16}^2 h_{18} - h_{16}^2 x_8y_{10} 
                - h_{16}x_8^3y_{10} - g_{24}x_8^2y_{10} 
                + y_{26}x_8^3)x_4^2 
\nonumber\\ && 
                +(h_{16}^2 y_{22} + h_{16}h_{18}x_{20} 
                + x_{20}y_{26}x_8 + y_{22}x_8^4)x_4
\nonumber\\ && 
                +(h_{16}x_{20}y_{22} + h_{18}x_{20}^2 
                + x_{20}y_{22}x_8^2),\nonumber\\
\label{eq:310}
  &y_{64} & = (y_{58}y_{10}-x_{20}y_{22}y_{26}-y_{22}^2x_8^3)/x_4\\
        &       & = (x_{36}x_8^2y_{10}^2 + x_{20}^2y_{22}y_{10} 
                + x_{20}y_{22}^2x_8 - y_{22}^2x_4x_8^3)/x_4^2\nonumber\\
        &       & = (h_{16}^2 h_{18}y_{10} - h_{16}^2 x_8y_{10}^2 
                + h_{16}h_{18}y_{26} - h_{16}x_8^3y_{10}^2 
                - g_{24}x_8^2y_{10}^2)x_4
\nonumber\\ && 
                +(y_{26}^2x_8 + y_{26}x_8^3y_{10})x_4
                +(h_{16}y_{22}y_{26} + h_{18}x_{20}y_{26} 
\nonumber\\ && \quad{} 
                - y_{22}y_{26}x_8^2 + y_{22}x_8^4y_{10}),\nonumber\\
\label{eq:311}
  &y_{60} & = (x_{36}x_8y_{10}^2  -  x_{20}y_{22}^2)/x_4\\
        &       & = ( - h_{16}h_{18}^2 - h_{16}^2 y_{10}^2 
                - h_{16}x_8^2y_{10}^2 - g_{24}x_8y_{10}^2 + y_{26}^2 
                + x_8^2y_{10}y_{26})x_4^2
\nonumber\\
 && 
                +(h_{16}h_{18}y_{22} + h_{18}^2x_{20} - y_{22}y_{26}x_8
                + y_{22}x_8^3y_{10})x_4
\nonumber
\end{eqnarray}
\begin{eqnarray} 
&& \quad{} 
                +(- h_{16}y_{22}^2 - h_{18}x_{20}y_{22} + y_{22}^2x_8^2).\nonumber\\
\label{eq:312}
        &y_{76} & = (y_{58}y_{22} + y_{60}x_{20} + x_{20}y_{22}^2x_8^2)/x_4\\
        &       & = (x_{36}x_{20}x_8y_{10}^2 + x_{36}y_{22}x_8^2y_{10}
                + x_{20}^2y_{22}^2 + x_{20}y_{22}^2x_4x_8^2)/x_4^2\nonumber\\
        &       & = ( - h_{16}^2 y_{10} - h_{16}x_8^2y_{10} - g_{24}x_8y_{10}
                + x_8^2y_{26})y_{26}x_4^2
\nonumber\\ && 
                + (h_{16}^2 h_{18}y_{22} - h_{16}h_{18}^2x_{20}
                + h_{16}^2 y_{22}x_8y_{10} + h_{16}y_{22}x_8^3y_{10})x_4
\nonumber\\ && 
                + (g_{24}y_{22}x_8^2y_{10} + x_{20}y_{26}^2)x_4
                + (h_{16}^2 y_{22}^2 - h_{16}h_{18}x_{20}y_{22}
                + h_{18}^2x_{20}^2 - y_{22}^2x_8^4).\nonumber
\end{eqnarray}

For a while, we denote by $A$ the algebra 
generated by the thirteen elements
\[
\{ x_4,x_8,y_{10},x_{20},y_{22},y_{26},x_{36},x_{48},x_{54},x_{58},y_{60},y_{64},y_{76} \} .
\]
Our aim is to prove that $A=H \cap \wtilde M = H^*(BT;\ZZ_3)^{W(E_6)}$.

We put $C = \ZZ_3[x_4, x_8, y_{10}]
                \{1, x_{20}, x_{20}^2, y_{22}, y_{22}^2, x_{20}y_{22}, 
                y_{58}, y_{60}, y_{76} \} \oplus 
             \ZZ_3[x_8, y_{10}]
                \{y_{26}, y_{26}^2,\linebreak 
                x_{20}y_{26}, y_{22}y_{26}, y_{64} \}
             $, where it is considered as a formal one.
Then we define a $\ZZ_3$-linear map %
$\sigma : C \otimes \ZZ_3[x_{36}, x_{48}, x_{54}] \to A$ by %
$\sigma(x^I\gamma_i \otimes y^J) = x^Iy^J\gamma_i$, 
where $x^I \in \ZZ_3[x_4, x_8, y_{10}]$, 
$y^J \in \ZZ_3[x_{36}, x_{48}, x_{54}]$ and $\gamma_i$ is an element
of $\{1, x_{20}, x_{20}^2, y_{22}, y_{22}^2, x_{20}y_{22}, 
                y_{58}, y_{60}, y_{76} \}$ and %
$\{y_{26}, y_{26}^2, x_{20}y_{26}, y_{22}y_{26}, y_{64} \}$.

\begin{Proposition}\label{37}
{\sl $\sigma : C \otimes \ZZ_3[x_{36}, x_{48}, x_{54}] \to A$ is
an isomorphism as a $\ZZ_3$-linear map. 
Hence its Poincar\'e polynomial PS is given by 
$$PS(A) = \Frac{g(t)}
        {(1-t^4)(1-t^8)(1-t^{10})(1-t^{36})(1-t^{48})(1-t^{54})}\, ,$$
where %
$g(t)= 1 +t^{20} +t^{22} +t^{26} -t^{30} +t^{40} +t^{42} +t^{44} +t^{46}
       +t^{48} -t^{50} -t^{56} +t^{58} +t^{60} +t^{64} -t^{68} +t^{76}$.
}
\end{Proposition}

To prove the proposition, we prepare some lemmas 

\begin{Lemma}\label{38}
{\sl $\sigma$ is injective.}
\end{Lemma}
\begin{proof}
We will show that the elements
\begin{equation}
\begin{array}{l}
A_1, A_2x_{20}, A_3x_{20}^2, A_4y_{22}, A_5y_{22}^2, A_6x_{20}y_{22}, 
A_7y_{58}, A_8y_{60}, A_9y_{76},\\
B_1y_{26}, B_2y_{26}^2, B_3x_{20}y_{26}, B_4y_{22}y_{26}, B_5y_{64} 
\end{array}
\label{eq:314}
\end{equation}
are linearly independent over $\ZZ_3$, where $A_i$ is a monomial 
of $x_4$, $x_8$, $y_{10}$, $x_{36}$, $x_{48}$, $x_{54}$ and %
$B_i$ a monomial of $x_8$, $y_{10}$, $x_{36}$, $x_{48}$, $x_{54}$.

To show it, we introduce a filtration 
to $H = \ZZ_3[t, x_4, x_8, y_{10}, h_{12}, h_{16}]$.
We define a weight $w$ by
$$ w(t)= w(h_{12}) = w(h_{16}) = 0, \quad
w(x_4)=1, \quad w(x_8)=2, \quad w(y_{10})= 3.$$
For a monomial of $H$, we define the weight by %
$w(x_4^{i_1} x_8^{i_2}y_{10}^{i_3}t^{i_4}h_{12}^{i_5}h_{16}^{i_6}) =
i_1w(x_4)+ i_2w(x_8)+ i_3w(y_{10})$.
For an element $x = \Sum \lambda_ix_i \in H$ \ $(\lambda_i \in \ZZ_3)$,
where $x_i$ is a monomial, 
we define that $\displaystyle w(x) = \inf_i w(x_i)$.
Then we introduce in $H$ by setting
$$F^pH = \{ x\in H \mid w(x) \ge p \}.$$
Since $F^{p+1}H$ is an ideal of $F^pH$ \ $(p \ge 0)$,
the associated graded
module $gr H = \mathop{\oplus}\limits_{p \ge 0}F^pH/F^{p+1}H$ is
an algebra. In $gr H$, we have
\begin{equation}
\begin{array}{l}
x_{20}=x_4h_{16}, \\ 
y_{22}=-x_4h_{18},  \quad (h_{18} = t^9 - t^3h_{12} + th_{16},) \\
y_{26}=x_8h_{18},\\
x_{36}=h_{12}^3, \quad 
x_{48}=h_{16}^3, \quad 
x_{54}=h_{18}^3, \\
y_{58}=x_4^2h_{16}^2 h_{18},  \quad 
y_{64}=x_4x_8h_{16}h_{18}^2,  \quad 
y_{60}=x_4^2h_{16}h_{18}^2,  \quad 
y_{76}=x_4^2h_{16}^2 h_{18}^2. 
\end{array}
\label{eq:315}
\end{equation}
For simplicity we denote $\ZZ_3[x_4, x_8, y_{10}, x_{36}, x_{48}, x_{54}]$ by $S$ and %
$\ZZ_3[x_8, y_{10}, x_{36}, x_{48}, x_{54}]$ by $T$.
Let 
$$
\begin{array}{l}
A_1 + A_2x_{20} + A_3x_{20}^2 + A_4y_{22} + A_5y_{22}^2 
+ A_6x_{20}y_{22} + A_7y_{58} + A_8y_{60} + A_9y_{76} \\
+ B_1y_{26} + B_2y_{26}^2 + B_3x_{20}y_{26} 
+ B_4y_{22}y_{26} + B_5y_{64}
\end{array}
\leqno(\ref{eq:315})'$$ 
equal to zero. Then we see that
$A_1 + A_2'h_{16} + A_3''h_{16}^2 + A_4'h_{18} + A_5''h_{18}^2 
+ A_6''h_{16}h_{18} + A_7''h_{16}^2 h_{18} + A_8''h_{16}h_{18}^2 
+ A_9''h_{16}^2 h_{18}^2 
+ \wtilde{B}_1h_{18} + \wtilde B_2h_{18}^2 
+ B_3'h_{16}h_{18} + B_4'h_{18}^2 + B_5'h_{16}h_{18}^2
=0$,
where $A_i \in S$, \ $A_i'=x_4A_i \in x_4S$, \ $A_i'' = x_4^2A_i 
\in x_4^2S$ %
\ and \ $B_i, \; \wtilde B_1=x_8B_1,\; \wtilde B_2=x_8^2B_2 \in T$, %
$B_i'=x_4x_8B_i \linebreak
\in x_4T$.

Noting $h_{18} = t^9 - t^3h_{12} + th_{16}$ in $gr H$, 
the elements $x_4$, $x_8$, $y_{10}$, $h_{12}$, $h_{16}$, $h_{18}$ are
algebraically independent over $\ZZ_3$.
So we obtain $A_1 = A_2' = A_3'' = A_7'' = A_8'' = 0$ at once.
With the respect to the coefficients of $h_{18}$, 
we note the fact $x_4S \cap T = \{ 0 \}$ and get $A_4' = \wtilde B_1=0$.
We can prove that the other coefficients are zero in a very similar way.
Hence we see that the elements indicated at (\ref{eq:314}) 
are 
linearly independent.
\end{proof}

Next we will prove that $\sigma$ is surjective. 
This part is the most crucial step in this section.
Before proving it, we observe the following lemma by a tedious
calculation.

\begin{Lemma}\label{39}
{\sl There are relations in $A$}:
$$\begin{array}{lll}
x_{54}x_8^3 &=& x_{48}y_{10}^3 - y_{26}^3 - y_{26}^2x_8^2y_{10}
          - y_{26}x_8^4y_{10}^2, \\
y_{58}x_8 &=&  - x_{48}x_4^2y_{10} + x_{20}^2y_{26} 
          + x_{20}y_{22}x_8^3
          + x_{20}y_{26}x_4x_8^2 + y_{22}x_4x_8^5 + y_{26}x_4^2x_8^4, \\
y_{58}y_{10} &=& y_{64}x_4 + x_{20}y_{22}y_{26} + y_{22}^2x_8^3, \\
y_{58}y_{22} &=&  - y_{60}x_{20} + y_{76}x_4 - x_{20}y_{22}^2x_8^2, \\
y_{58}y_{26} &=& x_{48}y_{22}x_4y_{10} + y_{64}x_{20}
          - x_{20}y_{22}y_{26}x_8^2 - y_{22}^2x_8^5 
          - y_{22}y_{26}x_4x_8^4, \\
y_{58}^2 &=&  - x_{48}y_{60}x_4^2 - x_{48}x_{20}^2x_4^2y_{10}^2
          -  x_{48}x_{20}y_{22}^2x_4 - x_{48}x_{20}x_4^3x_8^2y_{10}^2
          + y_{76}x_{20}^2 
    \\ && {}
          - x_{20}^3y_{22}y_{26}x_8 + x_{20}^3y_{26}^2x_4
          - x_{20}^3x_4x_8^4y_{10}^2 - x_{20}^2y_{22}^2x_8^4
          + x_{20}^2y_{22}y_{26}x_4x_8^3 
    \\ && {}
          + x_{20}^2x_4^2x_8^6y_{10}^2 - x_{20}^2y_{26}^2x_4^2x_8^2, \\
y_{58}y_{60} &=&  - x_{48}x_{54}x_4^4 - x_{48}y_{22}^3x_4
          - x_{48}y_{22}^2x_4^2x_8y_{10} + x_{48}y_{22}x_4^3x_8^2y_{10}^2
 \\ && {}
          + x_{48}y_{22}y_{26}x_4^3y_{10}
          - x_{48}y_{26}x_4^4x_8y_{10}^2 + x_{54}x_{20}^3x_4
          + x_{54}x_{20}^2x_4^2x_8^2 - y_{76}x_{20}y_{22} 
    \\ && {}
          - x_{20}^2y_{22}y_{26}^2x_4 - x_{20}y_{22}^2y_{26}x_4x_8^3
          + x_{20}y_{26}^3x_4^3x_8 - y_{22}^3x_4x_8^6
         + y_{22}^2x_4^2x_8^7y_{10} 
    \\ && {}
           - y_{22}^2y_{26}x_4^2x_8^5 - y_{22}y_{26}x_4^3x_8^6y_{10}
          + y_{22}y_{26}^2x_4^3x_8^4 + y_{26}^2x_4^4x_8^5y_{10}
          + y_{26}^3x_4^4x_8^3, \\
\end{array}
$$
$$
\begin{array}{lll}
y_{58}y_{64} &=& x_{48}x_{54}x_4^3x_8 + x_{48}y_{22}^3x_8
           - x_{48}y_{22}^2y_{26}x_4 - x_{48}y_{22}x_4^2x_8^3y_{10}^2 
    \\ 
&& {}
          + x_{48}y_{22}y_{26}x_4^2x_8y_{10} + x_{54}x_{20}^3x_8
          + x_{54}x_{20}^2x_4x_8^3 - x_{20}^2y_{22}y_{26}^2x_8 
    \\ && {}
          - x_{20}y_{22}^2y_{26}x_8^4 
         + y_{22}^3x_8^7 + y_{22}^2x_4x_8^8y_{10}
          + y_{22}^2y_{26}x_4x_8^6 + y_{22}y_{26}x_4^2x_8^7y_{10}, \\
y_{58}y_{76} &=& x_{48}x_{54}x_{20}x_4^3 - x_{48}y_{60}y_{22}x_4
          + x_{48}x_{20}y_{22}^3 + x_{48}y_{22}^2x_4^2x_8^3y_{10} 
    \\ && {}
          + x_{48}y_{22}y_{26}x_4^3x_8^2y_{10} + x_{54}x_{20}^4
          + x_{54}x_{20}^3x_4x_8^2 + x_{54}x_{20}^2x_4^2x_8^4 
    \\ && {}
          + x_{20}^2y_{22}^2y_{26}x_8^3
          - x_{20}^2y_{22}y_{26}^2x_4x_8^2 - x_{20}y_{22}^3x_8^6
          + x_{20}y_{22}^2y_{26}x_4x_8^5 
    \\ && {}
          - x_{20}y_{22}y_{26}^2x_4^2x_8^4
          - y_{22}^3x_4x_8^8
          + y_{22}^2y_{26}x_4^2x_8^7 - y_{22}y_{26}^2x_4^3x_8^6, \\
y_{60}x_8 &=& y_{64}x_4 - x_{20}y_{22}y_{26} + y_{22}^2x_8^3, \\
y_{60}y_{10} &=& x_{54}x_4^2x_8 + y_{22}^2x_8^2y_{10} - y_{22}^2y_{26}
          + y_{22}y_{26}x_4x_8^3y_{10}^2 - y_{22}x_4x_8y_{10}, \\
y_{60}y_{26} &=&  - y_{64}y_{22} + x_{54}x_{20}x_4x_8 
          + y_{22}^2x_8^4y_{10}
          + y_{22}y_{26}x_4x_8^3y_{10} - y_{22}y_{26}^2x_4x_8, \\
y_{60}^2 &=& x_{48}y_{22}^2x_4^2y_{10}^2 
          - x_{48}y_{22}x_4^3x_8y_{10}^3
          + x_{54}y_{58}x_4^2 + x_{54}x_{20}^2y_{22}x_4
          - x_{54}x_{20}y_{22}x_4^2x_8^2 
    \\ && {}
          + y_{76}y_{22}^2 + x_{20}y_{22}^3y_{26}x_8
          - x_{20}y_{22}^2y_{26}^2x_4 - y_{22}^4x_8^4
          - y_{22}^3x_4x_8^5y_{10} + y_{22}^3y_{26}x_4x_8^3 
    \\ && {}
          + y_{22}^2x_4^2x_8^6y_{10}^2 + y_{22}y_{26}x_4^3x_8^5y_{10}^2
          + y_{22}y_{26}^2x_4^3x_8^3y_{10} + y_{22}y_{26}^3x_4^3x_8, \\
y_{60}y_{64} &=&  - x_{48}x_{54}x_4^3y_{10} - x_{48}y_{22}^3y_{10}
          - x_{48}y_{22}^3y_{10} - x_{48}y_{22}^2x_4x_8y_{10}^2
          + x_{48}y_{22}x_4^2x_8^2y_{10}^3 
    \\ 
&& {}
          - x_{48}y_{26}x_4^3x_8y_{10}^3 - x_{54}x_{20}^2y_{22}x_8
          + x_{54}x_{20}^2y_{26}x_4 + x_{54}x_{20}y_{26}x_4^2x_8^2 
    \\ && {}
          + y_{22}^3x_8^6y_{10}
          - y_{22}^3y_{26}x_8^4 + y_{22}^2x_4x_8^7y_{10}^2
          - y_{22}^2y_{26}x_4x_8^5y_{10} 
          - y_{22}y_{26}x_4^2x_8^6y_{10}^2 
    \\ && {}
          - y_{22}y_{26}^2x_4^2x_8^4y_{10}
          - y_{22}y_{26}^3x_4^2x_8^2 + y_{26}^2x_4^3x_8^5y_{10}^2
          + y_{26}^3x_4^3x_8^3y_{10} + y_{26}^4x_4^3x_8, \\
y_{60}y_{76} &=&  - x_{48}x_{54}y_{22}x_4^3 - x_{48}y_{22}^4
          - x_{48}y_{22}^3x_4x_8y_{10} + x_{48}y_{22}^2y_{26}x_4^2y_{10} 
    \\ && {}
          + x_{48}y_{22}y_{26}x_4^3x_8y_{10}^2 + x_{54}y_{58}x_{20}x_4
          - x_{54}x_{20}^3y_{22} - x_{20}^2y_{22}^2y_{26}^2
          - x_{20}y_{22}^3y_{26}x_8^3 
    \\ && {}
          + x_{20}y_{22}^2y_{26}^2x_4x_8^2 - x_{20}y_{22}y_{26}^3x_4^2x_8
          - y_{22}^4x_8^6 - y_{22}^3x_4x_8^7y_{10}
          + y_{22}^2y_{26}x_4^2x_8^6y_{10} 
    \\ && {}
          - y_{22}y_{26}^2x_4^3x_8^5y_{10} - y_{22}y_{26}^3x_4^3x_8^3, \\
y_{64}x_8 &=&  - x_{48}x_4y_{10}^2 + x_{20}y_{26}^2 + y_{22}x_8^5y_{10}
          - y_{22}y_{26}x_8^3 + y_{26}x_4x_8^4y_{10} + y_{26}^2x_4x_8^2, \\
y_{64}y_{10} &=& x_{54}x_4x_8^2 + y_{22}x_8^4y_{10}^2
          - y_{22}y_{26}x_8^2y_{10} + y_{22}y_{26}^2, \\
y_{64}y_{26} &=& x_{48}y_{22}y_{10}^2 + x_{54}x_{20}x_8^2
          + y_{22}y_{26}^2x_8^2, \\
y_{64}^2 &=& x_{48}^2y_{26}^2x_4^2y_{10}^2
          - x_{48}x_{54}x_4^2x_8y_{10} + x_{48}y_{22}^2y_{26}y_{10}
          + x_{48}y_{22}x_4x_8^3y_{10}^3 
    \\ && {}
          - x_{48}y_{22}y_{26}x_4x_8y_{10}^2
          - x_{48}y_{26}x_4^2x_8^2y_{10}^3 + x_{54}x_{20}^2y_{26}x_8 
    \\ && {}
          + x_{20}y_{26}^4x_4 + y_{22}^2x_8^8y_{10}^2
          + y_{22}^2y_{26}x_8^6y_{10} - y_{22}^2y_{26}^2x_8^4
          - y_{22}y_{26}x_4x_8^7y_{10}^2  
    \\ && {}
          + y_{26}^2x_4^2x_8^6y_{10}^2
          - y_{26}^3x_4^2x_8^4y_{10}
          - y_{26}^4x_4^2x_8^2, \\
y_{64}y_{76} &=& x_{48}x_{54}y_{22}x_4^2x_8 - x_{48}x_{54}y_{26}x_4^3
          + x_{48}y_{22}^3y_{26} - x_{48}y_{22}^2x_4x_8^3y_{10}^2 
    \\ && {}
          - x_{48}y_{22}y_{26}x_4^2x_8^2y_{10}^2
          - x_{48}y_{22}y_{26}^2x_4^2y_{10}
          - x_{48}y_{26}^2x_4^3x_8y_{10}^2 
          + x_{54}x_{20}^3y_{26} 
    \\ && {}
          + x_{54}x_{20}^2y_{26}x_4x_8^2 + x_{20}^2y_{22}y_{26}^3
          - x_{20}y_{22}^2y_{26}^2x_8^3 + x_{20}y_{22}y_{26}^3x_4x_8^2 
    \\ && {}
          + x_{20}y_{26}^4x_4^2x_8 - y_{22}^3x_8^8y_{10}
          + y_{22}^3y_{26}x_8^6 - y_{22}^2y_{26}x_4x_8^7y_{10}
          - y_{22}^2y_{26}^2x_4x_8^5 
    \\ && {}
          + y_{22}y_{26}^2x_4^2x_8^6y_{10} - y_{22}y_{26}^3x_4^2x_8^4
          + y_{26}^3x_4^3x_8^5y_{10} + y_{26}^4x_4^3x_8^3, \\
y_{76}x_8 &=&  - x_{48}y_{22}x_4y_{10} + y_{64}x_{20}
          + x_{20}y_{22}y_{26}x_8^2 + y_{22}^2x_8^5 
          + y_{22}y_{26}x_4x_8^4, \\
y_{76}y_{10} &=& x_{54}x_{20}x_4x_8 + y_{64}y_{22}
          + y_{22}^2x_8^4y_{10} + y_{22}^2y_{26}x_8^2
          + y_{22}y_{26}x_4x_8^3y_{10} 
    \\ && {}
          - y_{22}y_{26}^2x_4x_8, \\
\end{array}
$$
$$
\begin{array}{lll}
y_{76}y_{26} &=& x_{48}y_{22}^2y_{10} + x_{54}x_{20}^2x_8
          - x_{20}y_{22}y_{26}^2x_8 + y_{22}^2y_{26}x_8^4
          + y_{22}y_{26}^2x_4x_8^3, \\
y_{76}^2 &=& x_{48}x_{54}x_{20}y_{22}x_4^2 + x_{48}y_{60}y_{22}^2
          + x_{54}y_{58}x_{20}^2 + x_{54}x_{20}^3y_{22}x_8^2
          + x_{20}^2y_{22}^2y_{26}^2x_8^2 
    \\ && {}
          + x_{20}y_{22}^3y_{26}x_8^5 + x_{20}y_{22}^2y_{26}^2x_4x_8^4
          + y_{22}^4x_8^8 - y_{22}^3y_{26}x_4x_8^7
          + y_{22}^2y_{26}^2x_4^2x_8^6.\\
\end{array}
$$
\end{Lemma}


To prove Proposition \ref{37},
we calculate the Poincar\'e polynomial of $A$.
From Lemmas \ref{38} and \ref{39}, we have
$$
PS(A) = PS(C \otimes \ZZ_3[x_{36}, x_{48}, x_{54}]) 
      = PS(C) \cdot PS(\ZZ_3[x_{36}, x_{48}, x_{54}]).
$$ 
Let $A\{s\}$ be a free $A$-module with the set of generators $S$.
Then we have $PS(A\{s\}) = PS(A) PS(\{s\})$.
Using this formula, we have
$$
PS(C) =
\Frac{1+t^{20}+t^{40}+t^{22}+t^{44}+t^{42}+t^{58}+t^{60}+t^{76}}
     {(1-t^4)(1-t^8)(1-t^{10})}
+
\Frac{t^{26}+t^{52}+t^{48}+t^{46}+t^{64}}{(1-t^8)(1-t^{10})}
.$$
We obtain the Poincar\'e polynomial of $A$ by a direct calculation.

\begin{Lemma}\label{310}
{\sl $\sigma$ is surjective.}
\end{Lemma}
\begin{proof}
The result of Lemma \ref{39} 
claims that $\Img \sigma$ is an algebra.
From the definition, $A$ is the minimal algebra of which contains 
the generators $x_4$, $x_8$, $y_{10}$, $x_{20}$, $y_{22}$, $y_{26}$, 
$x_{36}$, $x_{48}$, $x_{54}$, $y_{58}$, $y_{60}$, $y_{64}$, $y_{76}$.
It means that $A \subset \Img \sigma$.
But $\Img \sigma \subset A$ in obvious.
So we obtain that $\sigma$ is surjective.
\end{proof}

\begin{Proposition}\label{311}
{\sl We have 
the following presentation of $A$.}
$$\begin{array}{rl}
A = 
\Big(\ZZ_3[x_4, x_8, y_{10}]
        \{1, x_{20}, x_{20}^2, y_{26}, y_{26}^2, x_{20}y_{26}, 
                                        y_{58}, y_{64}, y_{76} \}&\\
\mbox{}\qquad\qquad
  \oplus 
  \ZZ_3[x_8, y_{10}]
        \{y_{22}, y_{22}^2, x_{20}y_{22}, y_{22}y_{26}, y_{60}\}&\Big)
\otimes 
  \ZZ_3[x_{36}, x_{48}, x_{54} ].
\end{array}$$
\end{Proposition}

\begin{proof}
We show that the elements shown at the proposition  are 
linear independent. To do this, we introduce a filtration 
into $H = \ZZ_3[t, x_4, x_8, y_{10}, h_{12}, h_{16}]$.  
We define a weight $v$ by %
$v(t)=1$, \ $v(x_4)=+\infty$, \ $v(x_8)=6$, \ $v(y_{10})= 7$, %
\ $v(h_{12}) = v(h_{16}) = 9$. Then we can define a filtration to $H$ as
defined in the proof of Lemma \ref{38}. 
It is immediate from the definition 
that at $gr H$ we have
\begin{equation}
\begin{array}{l}
\label{eq:315}
x_{20} \equiv x_8h_{12}, \quad 
y_{22} \equiv y_{10}h_{12},  \quad 
y_{26} \equiv x_8h_{18},\\
y_{58} \equiv x_8^2h_{12}^2 h_{18},  \quad 
y_{60} \equiv x_8y_{10}h_{12}^2h_{18},  \quad 
y_{64} \equiv x_8^2h_{12}h_{18}^2,  \quad 
y_{76} \equiv x_8x_4^2h_{12}^2 h_{18}^2, \\
x_{36} \equiv h_{12}^3, \quad 
x_{48} \equiv h_{16}^3, \quad 
x_{54} \equiv h_{18}^3.
\end{array}
\end{equation}
As is Lemma \ref{38} 
shown, we can show 
that the monomials in the above presentation are linearly independent.

We calculate the Poincar\'e polynomial of the presentation.
It is given by

\mathvskip\mbox{}\qquad
$
\Bigg(
\Frac{1+t^{20}+t^{40}+t^{22}+t^{44}+t^{42}+t^{58}+t^{60}+t^{76}}
     {(1-t^4)(1-t^8)(1-t^{10})}
+
\Frac{t^{26}+t^{52}+t^{48}+t^{46}+t^{64}}{(1-t^8)(1-t^{10})}\Bigg)\\[1ex]
\mbox{}\hfill
\times \Frac1{(1-t^{36})(1-t^{48})(1-t^{54})}
.\hspace{3em}
$

\mathvskip\noindent
We see that it coinsides with that of $A$.
Hence we have completed the proof.
\end{proof}

\begin{Lemma}\label{312}
{\sl Let
$$\ZZ_3[x_4, x_{36}, x_{48}, x_{54}] \otimes 
\{ \ZZ_3[x_8, y_{10}]\{\sigma_i \mid 1 \le i \le 9\} \oplus
\ZZ_3[y_{10}] \{\beta_j \mid 1 \le j \le 5\} \}$$
be the presentation shown at Proposition $3.11.$ If we assume that 
$$ \Sum_{i=1}^{9} f_i \cdot g_i \sigma_i
 + \Sum_{j=1}^5 f_j \cdot g_j \beta_j
 \equiv 0  \quad \mod x_4 \ZZ_3[t, x_4, x_8, y_{10}, h_{12}, h_{16}],$$
where $f_i, f_j \in \ZZ_3[x_4, x_{36}, x_{48}, x_{54}]$, 
           $g_i \in \ZZ_3[x_8, y_{10}]$, 
           $g_j \in \ZZ_3[y_{10}]$, 
then $f_{\lambda} \equiv 0 
\ \mod x_4 \ZZ_3[t,\linebreak x_4, x_8, y_{10}, h_{12}, h_{16}]$ holds 
for $\lambda = i,j$.
}
\end{Lemma}

\bigskip
{\bf Proof of Lemma \ref{312}. 
}\quad 
It is sufficient that we prove the statement in 
$\mathcal{G}rH$.
Hence it is enough to prove that
$$ \Sum_{i=1}^9 g_i(x_8^3, y_{10}^3) \wtilde\sigma_i
 + \Sum_{j=1}^5 g_j(y_{10}^3) \wtilde\beta_j \equiv 0 
                                        \quad \mod x_4\,gr H$$
implies $g_i = g_j =0$.

From (\ref{eq:315}), 
all the $\wtilde\sigma_i$ and $\wtilde\beta_j$ contain
no the $x_4$-factor. Then we can prove the statement in a similar way 
to the first part of Proposition \ref{311}.

\begin{Theorem}\label{313}
$(1)$ \enskip {\sl
$H^*(BT; \ZZ_3)^{W(E_6)}$ is generated by the thirteen elements
$$x_4,\  x_8,\  y_{10},\  x_{20},\  y_{22},\  y_{26},\  
x_{36},\  x_{48},\  x_{54},\  y_{58},\  y_{60},\  y_{64},\  y_{76},$$
where these elements are defined from {\rm (\ref{eq:37})
} to {\rm (\ref{eq:312})
}.}

$(2)$ \enskip {\sl The relations are given at Lemma {\rm \ref{39}
}}
\end{Theorem}

\begin{proof}
From Theorem \ref{36},
we have 
$$H^*(BT;\ZZ_3)^{W(E_6)} = H \cap \wtilde M,$$ 
where $
  H        = \ZZ_3[\ts t, x_4, x_8, y_{10}, h_{12}, h_{16}]$, 
$\wtilde M = \ZZ_3[x_4,x_4^{-1}][x_8, y_{10}, x_{20}, y_{22}, x_{36}]$.
Let $A$ be an algebra generated by $x_4$, \ldots, $y_{76}$.
Then $\ZZ_3[x_4,x_4^{-1}] \mathop{\otimes}\limits_{\sZZ_3[x_4]} A$ %
contains $\wtilde M$. Hence it is enough to prove that %
$H \cap 
\Big(\ZZ_3[x_4,x_4^{-1}] \mathop{\otimes}\limits_{\sZZ_3[x_4]} A \Big) = A$. %

In other word, when we use the presentation given at Proposition \ref{311}, 
it is sufficient to show that 
$$ \Sum_{i=1}^9 f_i \cdot g_i(x_8^3, y_{10}^3) \sigma_i
 + \Sum_{j=1}^5    f_j \cdot g_j(y_{10}^3) \beta_j \equiv 0 
                                                \quad \mod x_4 H,$$
$f_i, f_j \in \ZZ_3[x_4, x_{36}, x_{48}, x_{54}]$ 
implies $f_{\lambda} \equiv 0 \mod x_4 H$ for all $\lambda = i,j$.

It is just the same as the statement of Lemma \ref{312}. 
Hence we have completed the proof.
\end{proof}
\section{Appendix}
 From now on, coefficients will be in  $\ZZ_3^{}$  throughout the
calculation.

Denote by $c\ts_i^{}
$  %
and  $p\ts_i^{} 
$
the elementary symmetric functions on  $\{t\ts_i^{}\}$  and  $\{t\ts_j^2\}$, respectively.
Then we have
\begin{equation}
\begin{array}[t]{llllllll}
\label{eq:a11}
& c_2^{}       & = p_1^{} - c_1^2, \\
& c_4^{}       & = {}- p_2^{} + c_1^4 + c_1^2p_1^{} + c_1^{}c_3^{} + p_1^2, \\
& c_3^{}c_5^{} & = p_4^{} - c_4^2, \\
& c_3^2        & = p_3^{} + c_1^6 + c_1^3c_3^{} - c_1^2p_2^{} 
		- c_1^{}c_3^{}p_1^{} + c_1^{}c_5^{} - p_1^3 + p_1^{}p_2^{}.
\end{array}
\end{equation}

The first six $W(E_6^{})$-invariant elements \ $x_4^{}$,\ $x_8^{}$,\ $y_{10}^{}$,\ $x_{20}^{}$,\ $y_{22}^{}$ and \ $y_{26}^{}$ \ are easily found as is in Section 1: 

\begin{equation}
\begin{array}[t]{llllllll}
\label{eq:a12}
&x_4 = p_1^{},\\ 
&x_8 = p_2^{} - p_1^2 ,\\ 
&y_{10} = c_5^{} - x_8^{} t - x_4^{} t^3,\\
&x_{20}^{} = h_{12}^{}x_8^{} + h_{16}^{}x_4^{},\\
&y_{22}^{} = h_{12}^{}y_{10}^{} - h_{18}^{}x_4^{},\\
&y_{26}^{} = h_{16}^{}y_{10}^{} + h_{18}^{}x_8^{},
\end{array}
\end{equation}
where
%
\begin{equation}
\begin{array}[t]{llllllll}
\label{eq:a13}
&h_{12} = p_3^{} + y_{10}^{} t - x_8^{} t^2 - x_4^{} t^4,\\
&h_{16} = p_4^{} + y_{10}^{} t^3 - x_8^{} t^4 - x_4^{} t^6,\\
&h_{18}^{} = t \ts (h_{16}^{} - x_8^2) + t^3( - h_{12}^{} + x_4^{}x_8^{}) 
		+ t^5x_8^{} - t^7x_4^{} + t^9.
\end{array}
\end{equation}

Observe that there holds the following relation:
\begin{equation}
\label{eq:a14}- x_{20}^{}y_{10}^{} + y_{22}^{}x_8^{} + y_{26}^{}x_4^{} = 0.
\end{equation}

{ Calculation of $\sigma_j^S$ is carried out as follows.}

Put $w_1 = t - c_1^{}$ and $w_i = t + c_1^{} - t_{i-1}^{}$ for $2 \le i \le 6$.
The set  $$S = \{w_i^{} + w_j^{},\ c_1^{} - w_i^{}, \ {}- c_1^{} - w_i^{} \ ; \ i<j \}$$
is invariant as a set under the action of  $W(E_6^{})$ (see \cite{TW}).  
Therefore the elementary symmetric functions $\sigma_j^S$  on $S$ are 
invariant under the action of  $W(E_6^{})$.

Put

\begin{equation}
\begin{array}{llllllll}
\label{eq:a21}
&P &= 1 +  \Sum_{j=1}^{27} \sigma_j^S =  \Prod_{y \in S}\; (1 + y)\\
&  &=   \Prod_{1 \le i < j \le 6} \;(1 + w_i^{} + w_j^{}) 
	\Prod_{1 \le j \le 6}(1 + c_1^{} - w_j^{}) 
	\Prod_{1 \le j \le 6}(1 - c_1^{} - w_j^{}).
\end{array}
\end{equation}

(We shall rewrite $P$ as far as possible in terms of the six invariant elements in (\ref{eq:a12})%
.)

$P$ is expressed by $t$ and $c_i^{}$'s and has 2600 terms.
We replace $c_2^{}$, $c_4^{}$, $c_5^{}$, $c_3^2$ and $c_3^{}y_{10}^{}$ %
in  $P$  by the following relations:

\begin{eqnarray}
\label{eq:a22}
 c_2^{} & = & {}- c_1^2 + x_4^{},\\   
 c_4^{} & = & c_1^4 + c_1^{} c_3^{} + c_1^2 x_4^{} - x_8^{},\nonumber\\
 c_5^{} & = & y_{10}^{} + x_8^{} t + x_4^{} t^3,\nonumber\\
\label{eq:a23}
 c_3^2 & = & h_{12}^{} + t^4x_4^{} + t^2x_8^{} - t \ts y_{10}^{} + x_4^{}x_8^{}
    + t^3x_4^{}c_1^{} + t \ts x_8^{}c_1^{}
 - x_4^2c_1^2 \\
&&- x_4^{}c_1^{}c_3^{} 
    - x_8^{}c_1^2  {}+ y_{10}^{}c_1^{} + c_1^6 + c_1^3c_3^{},\nonumber
\end{eqnarray}
\begin{eqnarray}
\label{eq:a24}
 c_3^{}y_{10}^{} & = & h_{16}^{} + t^6x_4^{} + t^4x_8^{} - t^3y_{10}^{} - x_8^2  
    - t^4x_4^{}c_1^2 - t^3x_4^{}c_1^3\\ && - t^3x_4^{}c_3^{} - t^2x_8^{}c_1^2
  - t \ts x_8^{}c_1^3 
    - t \ts x_8^{}c_3^{} + t \ts y_{10}^{}c_1^2 
    - h_{12}^{}c_1^2 + x_4^{}x_8^{}c_1^2 \nonumber \\ &&+ x_4^{}c_1^6 - x_4^{}c_1^3c_3^{} 
    - x_8^{}c_1^{}c_3^{}
    - y_{10}^{}c_1^3 + c_1^8\nonumber
\end{eqnarray}

and we see that all $c_3^{}$'s are cancelled.

We use two more relations to eliminate $c_1^{}$'s. 

Replacing $c_3^2$ in the equality $c_3^2y_{10} - c_3\cdot c_3^{}y_{10}^{} 
= 0$, we get a relation:

\begin{equation}
\begin{array}{llllllll}
\label{eq:a25}
& h_{16}^{}c_3^{} = & t^9x_4^{} + t^7x_4^2 + t^7x_8^{} - t^6y_{10}^{} 
    - t^5x_4^{}x_8^{} + t^3h_{12}^{}x_4^{} + t^3h_{16}^{} + t^3x_4^2x_8^{}
    + th_{12}^{}x_8^{} 
\\ &&
    + t \ts x_4^{}x_8^2 - t \ts y_{10}^2  + h_{12}^{}y_{10}^{} + x_4^{}x_8^{}y_{10}^{}  
    + t^7x_4^{}c_1^2 + t^6x_4^{}c_1^3
    + t^6x_4^{}c_3^{} + t^5x_4^{}c_1^4 + t^5x_8^{}c_1^2 
\\ &&
    - t^4x_4^2c_1^3 - t^4x_4^{}x_8^{}c_1^{} - t^4x_4^{}c_1^5 - t^4x_4^{}c_1^2c_3^{}
    + t^4x_8^{}c_1^3 + t^4x_8^{}c_3^{} - t^4y_{10}^{}c_1^2 - t^3h_{12}^{}c_1^2 
\\ &&
    - t^3x_4^3c_1^2 - t^3x_4^2c_1^4 + t^3x_4^{}x_8^{}c_1^2
    - t^3x_4^{}c_1^3c_3^{} + t^3x_8^{}c_1^4 - t^3x_8^{}c_1^{}c_3^{} 
    + t^3c_1^8 - t^2x_4^{}x_8^{}c_1^3 
\\ &&
    - t^2x_8^2c_1^{} - t^2x_8^{}c_1^5
    - t^2x_8^{}c_1^2c_3^{} - t^2y_{10}^{}c_1^4 + t \ts h_{12}^{}c_1^4 
    - t \ts h_{16}^{}c_1^2 - t \ts x_4^2x_8^{}c_1^2 + t \ts x_4^{}x_8^{}c_1^4
\\ &&
    + t \ts x_4^{}y_{10}^{}c_1^3 - t \ts x_4^{}c_1^8 + t \ts x_4^{}c_1^5c_3^{} 
    + t \ts x_8^2c_1^2 + t \ts x_8^{}y_{10}^{}c_1^{} - t \ts x_8^{}c_1^6 + t \ts x_8^{}c_1^3c_3^{}
    - t \ts c_1^{10} 
\\ &&
    - h_{12}^{}x_4^{}c_1^3 + h_{12}^{}x_8^{}c_1^{} + h_{12}^{}c_1^5 
    + h_{12}^{}c_1^2c_3^{} - h_{16}^{}x_4^{}c_1^{} - h_{16}^{}c_1^3 - x_4^3c_1^5
    - x_4^2x_8^{}c_1^3 
\\ &&
    - x_4^2y_{10}^{}c_1^2 - x_4^2c_1^7 - x_4^{}x_8^2c_1^{} 
    + x_4^{}x_8^{}c_1^5 - x_4^{}x_8^{}c_1^2c_3^{} - x_4^{}y_{10}^{}c_1^4
    - x_4^{}c_1^9 + x_4^{}c_1^6c_3^{} 
\\ &&
    + x_8^2c_3^{} + x_8^{}c_1^7 - x_8^{}c_1^4c_3^{} + y_{10}^2 c_1^{} 
    - y_{10}^{}c_1^6 - c_1^{11} - c_1^8c_3^{}.
\end{array}
\end{equation}

Making use of rewriting $c_3^2$, $c_3y_{10}$ and $c_3^{}h_{16}^{}$, 
the equality $c_3^2y_{10}^2 - (c_3^{}y_{10}^{})^2 = 0$  gives rise 
to the following relation.

\begin{equation}
\begin{array}{llllllll}
\label{eq:a26}
\end{array}
\begin{array}{llllllll}
& c_1^{16} = & {}- t^{12}x_4^2 + t^{10}x_4^3 + t^{10}x_4^{}x_8^{} - t^9x_4^{}y_{10}^{} 
    - t^8x_8^2 + t^7x_4^2y_{10}^{} - t^7x_8^{}y_{10}^{}
    + t^6h_{12}^{}x_4^2 
\\ & & {}    
    + t^6h_{16}^{}x_4^{} + t^6x_4^3x_8^{} - t^6x_4^{}x_8^2 - t^6y_{10}^2  
    - t^5x_4^{}x_8^{}y_{10}^{} - t^4h_{12}^{}x_4^{}x_8^{}
    + t^4h_{16}^{}x_8^{} 
\\ & & {}     
    - t^4x_4^2x_8^2 - t^4x_4^{}y_{10}^2  - t^3h_{12}^{}x_4^{}y_{10}^{} 
    - t^3h_{16}^{}y_{10}^{} - t^3x_4^2x_8^{}y_{10}^{}
    - t^3x_8^2y_{10}^{} + t^2h_{12}^{}x_8^2 
\\ & & {}     
    + t^2x_4^{}x_8^3 - t^2x_8^{}y_{10}^2  - t \ts h_{12}^{}x_8^{}y_{10}^{} 
    - t \ts x_4^{}x_8^2y_{10}^{} - t \ts y_{10}^3 
    + h_{12}^{}y_{10}^2  - h_{16}^2  - h_{16}^{}x_8^2 
\\ & & {}     
    + x_4^{}x_8^{}y_{10}^2  - x_8^4 - t^{10}x_4^2c_1^2 - t^8x_4^2c_1^4 
    + t^8x_4^{}x_8^{}c_1^2
    - t^7x_4^{}y_{10}^{}c_1^2 - t^6h_{12}^{}x_4^{}c_1^2  
\\ & & {}     
    - t^6x_4^4c_1^2 - t^6x_4^3c_1^4 + t^6x_4^2x_8^{}c_1^2 + t^6x_4^2c_1^6
    - t^6x_4^{}x_8^{}c_1^4 + t^6x_4^{}c_1^8 - t^6x_8^2c_1^2 
\\ 
\end{array}
\end{equation}
\begin{equation*}
\begin{array}{llllllll}
& & {}     
    - t^5x_4^{}y_{10}^{}c_1^4 - t^5x_8^{}y_{10}^{}c_1^2 + t^4h_{12}^{}x_4^{}c_1^4
    - t^4h_{12}^{}x_8^{}c_1^2 - t^4h_{16}^{}x_4^{}c_1^2 + t^4x_4^3x_8^{}c_1^2 
\\ 
& & {}     
    - t^4x_4^3c_1^6 + t^4x_4^2x_8^{}c_1^4 
    + t^4x_4^2c_1^8
    + t^4x_4^{}x_8^2c_1^2 - t^4x_4^{}x_8^{}c_1^6 - t^4x_4^{}c_1^{10} + t^4x_8^{}c_1^8 
\\ 
& & {}     
    - t^4y_{10}^2 c_1^2 - t^3h_{12}^{}x_4^{}x_8^{}c_1^{}
    + t^3h_{12}^{}y_{10}^{}c_1^2 - t^3h_{16}^{}x_4^2c_1^{} 
    + t^3x_4^4c_1^5 - t^3x_4^3x_8^{}c_1^3 
\\ 
& & {}     
    + t^3x_4^3y_{10}^{}c_1^2 + t^3x_4^3c_1^7
    + t^3x_4^2x_8^{}c_1^5 + t^3x_4^2y_{10}^{}c_1^4 - t^3x_4^{}x_8^2c_1^3 
    - t^3x_4^{}x_8^{}y_{10}^{}c_1^2 
\\ & & {}
    + t^3x_4^{}x_8^{}c_1^7 + t^3x_4^{}y_{10}^2 c_1^{} 
    - t^3x_4^{}y_{10}^{}c_1^6 + t^3x_8^{}y_{10}^{}c_1^4 - t^3y_{10}^{}c_1^8 
    + t^2h_{12}^{}x_8^{}c_1^4
\\ 
& & {}     
    - t^2h_{16}^{}x_8^{}c_1^2 - t^2x_4^2x_8^2c_1^2 - t^2x_4^2x_8^{}c_1^6 
    - t^2x_4^{}x_8^2c_1^4 + t^2x_4^{}x_8^{}c_1^8 + t^2x_8^2c_1^6
    - t^2x_8^{}c_1^{10} 
\\ & & {}     
    - t^2y_{10}^2 c_1^4 - t \ts h_{12}^{}x_8^2c_1^{} - t \ts h_{12}^{}y_{10}^{}c_1^4 
    - t \ts h_{16}^{}x_4^{}x_8^{}c_1^{}
    + t \ts h_{16}^{}y_{10}^{}c_1^2 + t \ts x_4^3x_8^{}c_1^5 
\\ 
& & {}     
    - t \ts x_4^2x_8^2c_1^3 + t \ts x_4^2x_8^{}y_{10}^{}c_1^2 
    + t \ts x_4^2x_8^{}c_1^7 + t \ts x_4^2y_{10}^{}c_1^6
    + t \ts x_4^{}x_8^2c_1^5 + t \ts x_4^{}x_8^{}y_{10}^{}c_1^4 
\\ & & {}     
    - t \ts x_4^{}y_{10}^{}c_1^8 
    - t \ts x_8^3c_1^3 + t \ts x_8^2c_1^7 + t \ts x_8^{}y_{10}^2 c_1^{}
    - t \ts x_8^{}y_{10}^{}c_1^6 + t \ts y_{10}^{}c_1^{10} 
    - h_{12}^2 c_1^4 - h_{12}^{}h_{16}^{}c_1^2 
\\ & & {}     
    - h_{12}^{}x_4^2c_1^6 - h_{12}^{}x_4^{}x_8^{}c_1^4
    + h_{12}^{}x_4^{}c_1^8 - h_{12}^{}x_8^2c_1^2 - h_{12}^{}x_8^{}y_{10}^{}c_1^{} 
    - h_{12}^{}x_8^{}c_1^6 - h_{12}^{}c_1^{10} 
\\ 
& & {}     
    - h_{16}^{}x_4^2c_1^4 - h_{16}^{}x_4^{}y_{10}^{}c_1^{} - h_{16}^{}x_4^{}c_1^6 
    + h_{16}^{}x_8^{}c_1^4 + h_{16}^{}c_1^8 
    - x_4^4c_1^8 + x_4^3x_8^{}c_1^6 + x_4^3y_{10}^{}c_1^5
\\ & & {}     
    - x_4^3c_1^{10} - x_4^2x_8^{}y_{10}^{}c_1^3 - x_4^2y_{10}^2 c_1^2 
    + x_4^2y_{10}^{}c_1^7 + x_4^{}x_8^3c_1^2 
    + x_4^{}x_8^{}y_{10}^{}c_1^5 - x_4^{}c_1^{14} + x_8^3c_1^4 
\\ & & {}     
- x_8^2y_{10}^{}c_1^3 + x_8^{}y_{10}^2 c_1^2 
    + x_8^{}y_{10}^{}c_1^7 + x_8^{}c_1^{12} + y_{10}^3 c_1^{} - y_{10}^2 c_1^6.
\end{array}
\end{equation*}

Observe that there is no $c_3^{}$ in (\ref{eq:a26}).

Making use of (\ref{eq:a26}), 
we find that $c_1$ is cancelled out and %
$P \in  \ZZ_3^{}[t,x_4^{},x_8^{},y_{10}^{},h_{12}^{},h_{16}^{}]$.
The highest degree of \ $t$ \ is 27.

From (\ref{eq:a13}) and (\ref{eq:a14})
the following rewriting rules are of use.
\begin{eqnarray}
\label{eq:a27} & t^9& = h_{18}^{} + t^7x_4^{} - t^5x_8^{} + t^3h_{12}^{} - t^3x_4^{}x_8^{} 
 	- t \ts h_{16}^{} + t \ts x_8^2,\\
\label{eq:a28}& h_{18}^{}x_8^{}& = {}- h_{16}^{}y_{10}^{} + y_{26}^{}, \\
& h_{18}^{}x_4^{} &= h_{12}^{}y_{10}^{} - y_{22}^{}, \nonumber\\
& h_{16}^{}x_4^{} &= {}- h_{12}^{}x_8^{} + x_{20}^{}, \nonumber\\
& y_{26}^{}x_4^{}& = x_{20}^{}y_{10}^{} - y_{22}^{}x_8^{},\nonumber
\end{eqnarray}
and, in case of need
\begin{eqnarray}
& h_{18}^{}x_{20}^{} &= h_{12}^{}y_{26}^{} - h_{16}^{}y_{22}^{}\hspace{5.8cm}\nonumber.
\end{eqnarray}

     Now we decompose $P$ by degree and see that $\sigma_j^S$  for $j=18, 24, 27$ 
are not expressed  by $x_4^{}$,\ $x_8^{}$,\ $y_{10}^{}$,\ $x_{20}^{}$,\ $y_{22}^{}$ and \ $y_{26}^{}$.

     The element  $\sigma_{18}^S$   is now  as follows.

$\begin{array}{llllllll}
\hspace*{2em}
& \sigma_{18}^S & = & {} - x_{20}^{}x_4^2x_8^{} - x_{20}^{}x_8^2 - y_{22}^{}x_4^{}y_{10}^{} 
	- y_{26}^{}y_{10}^{} - x_4^3x_8^3 + x_4^2x_8^{}y_{10}^2 + x_4^{}x_8^4 - x_8^2y_{10}^2
\\&&&{}
	- h_{12}^3 + h_{12}^2x_4^{}x_8^{} + h_{12}^{}x_{20}^{}x_4^{} - h_{12}^{}x_4^2x_8^2  
    - t^8x_4^3x_8^{} + t^7x_4^3y_{10}^{} - t^6x_4^6 + t^6x_4^4x_8^{} 
\\&&&{}
    + t^4x_4^5x_8^{} - t^4x_4^3x_8^2 + t^3h_{12}^{}x_4^2y_{10}^{} 
    - t^3y_{22}^{}x_4^2 + t^3x_4^5y_{10}^{} - t^3x_4^3x_8^{}y_{10}^{} 
\\&&&{}
    + t^2x_{20}^{}x_4^3 - t^2x_4^3y_{10}^2 
    - t \ts y_{22}^{}x_4^3 - t \ts x_4^4x_8^{}y_{10}^{}.
\end{array}$

Collect the terms with $t$ in  $\sigma_{18}^S$  and we put
\begin{equation}
\begin{array}{llllllll}
\label{eq:a29}
& g_{24}^{} & = & {} - t^8x_8^{} + t^7y_{10}^{} - t^6x_4^3 + t^6x_4^{}x_8^{} 
	+ t^4x_4^2x_8^{} - t^4x_8^2 + t^3h_{18}^{} + t^3x_4^2y_{10}^{} - t^3x_8^{}y_{10}^{} 
\\&&&{}
     + t^2x_{20}^{} - t^2y_{10}^2 - t \ts y_{22}^{} - t\ts x_4^{}x_8^{}y_{10}^{}
\end{array}
\end{equation}
then define
\begin{equation}
\begin{array}{llllllll}
\label{eq:a210}
& x_{36}^{} = {} - g_{24}^{}x_4^3 + h_{12}^3 - h_{12}^2x_4^{}x_8^{} - h_{12}^{}x_{20}^{}x_4^{} + h_{12}^{}x_4^2x_8^2,\\
& x_{48} = g_{24}^{} x_8^3+ h_{16}^3+h_{16}^2 x_8^2+h_{16} x_8^4,\\
& x_{54} = {}- g_{24}^{} y_{10}^3+h_{18}^3 - h_{16}^{} h_{18}^{} y_{10}^2+h_{18}^2 x_8^{} y_{10}^{} + h_{18}^{} x_8^2 y_{10}^2.
\end{array}
\end{equation}

     The elementary symmetric functions on $S$ are calculated as follows:

\begin{equation}
\begin{array}{llllllll}
\label{eq:a211}
& \sigma_i^S & = & 0 \quad \hbox{for} \quad  1 \le i \le 5  
			\quad \hbox{and} \quad  i = 7,{10},{11},{13},{19}, \\
& \sigma_6^S & = & {}- x_4^3 - x_4^{}x_8^{}, \\
& \sigma_8^S & = & {}- x_4^2x_8^{} - x_8^2, \\
& \sigma_9^S & = & {}- x_4^2y_{10}^{} - x_8^{}y_{10}^{}, \\
& \sigma_{12}^S & = &  - x_{20}^{}x_4^{} + x_4^6 - x_4^4x_8^{} + x_4^{}y_{10}^2 - x_8^3, \\
& \sigma_{14}^S & = & x_{20}^{}(x_4^2 - x_8^{}) - x_4^5x_8^{} 
			- x_4^3x_8^2 - x_4^2y_{10}^2  + x_8^{}y_{10}^2 , \\
& \sigma_{15}^S & = & x_{20}^{}y_{10}^{} + y_{22}^{}(x_4^2 + x_8^{}) 
		- x_4^5y_{10}^{} + x_4^{}x_8^2y_{10}^{} - y_{10}^3 , \\
& \sigma_{16}^S & = &  - x_{20}^{}x_4^3 + x_4^3y_{10}^2 , \\
& \sigma_{17}^S & = & x_{20}^{}x_4^{}y_{10}^{} 
		+ y_{22}^{}(x_4^3 - x_4^{}x_8^{}) - y_{26}^{}x_8^{} 
		\\
&&&+ x_4^4x_8^{}y_{10}^{} - x_4^2x_8^2y_{10}^{} + x_4^{}y_{10}^3  
		+ x_8^3y_{10}^{},\\
& \sigma_{18}^S & = & 
{}-x_{36} -x_{20}^{} x_4^2 x_8^{} -x_{20}^{} x_8^2 -x_4^3 x_8^3\\
&&&
+x_4^{} x_8^4 +x_4^2 x_8^{} y_{10}^2 -x_8^2 y_{10}^2
-x_4^{} y_{10}^{} y_{22}^{} -y_{10}^{} y_{26}^{},\\
& \sigma_{20}^S & = & 
{}-x_{20}^2 +x_{20}^{} x_4^{} x_8^2 +x_4^2 x_8^4 +x_8^5
-x_{20}^{} y_{10}^2 -x_4^{} x_8^2 y_{10}^2 -y_{10}^4,\\
& \sigma_{21}^S & = & 
x_{20}^{} x_4^{} x_8^{} y_{10}^{} + x_4^3 y_{10}^3 +x_4^{} x_8^{} y_{10}^3
-x_{20}^{} y_{22}^{} -x_4^{} x_8^2 y_{22}^{} +y_{10}^2 y_{22}^{} +x_8^2 y_{26}^{},\\
& \sigma_{22}^S & = & 
{}-x_{20}^{} x_8^3 +x_8^3 y_{10}^2,\\
& \sigma_{23}^S & = & 
{}-x_4^2 x_8^{} y_{10}^3 -x_8^2 y_{10}^3 +x_{20}^{} y_{26}^{} -y_{10}^2 y_{26}^{},\\
& \sigma_{24}^S & = & 
x_{48}^{} -x_{20}^{} x_8^{} y_{10}^2
+x_4^{} x_8^3 y_{10}^2 - x_4^2 y_{10}^4 + x_8^{} y_{10}^4 + y_{22}^{} y_{26}^{},\\
\end{array}
\end{equation}

$\begin{array}{llllllll}
\hphantom{(2.11)}\hspace{1.2em}  
& \sigma_{25}^S & = & 
{}-x_{20}^{} y_{10}^3 +y_{10}^5,\\
& \sigma_{26}^S & = & 
{}-x_8^4 y_{10}^2 + x_4^{}x_8^{}y_{10}^4 + y_{10}^3y_{22}^{} 
- x_8^2 y_{10}^{}y_{26}^{} -y_{26}^2,\\
& \sigma_{27}^S & = & 
{}-x_{54}.
\end{array}$

\bigskip
Note that the elements $x_4^{}$,\ $x_8^{}$,\ $y_{10}^{}$,\ $x_{20}^{}$,\ $y_{22}^{}$, \ $y_{26}^{}$, \ $x_{36}^{}$, \ $x_{48}^{}$ and $x_{54}^{}$ are also found 
in Section 3 by making use of Galois theory.




\begin{thebibliography}{9999}
\bibitem[AR]{AR}  S. Araki,  
  {\em On the non-commutativity of Pontrjagin rings mod $3$
      of some compact exceptional groups,}  
      Nagoya Math. J., {\bf 17} (1960), 225\ts--260.

\bibitem[A]{A}   E. Artin,  
  {\em Galois Theory,}  Notre Dame Math. Lectures {\bf 2}, 1979.

\bibitem[BH]{BH}  A. Borel and F. Hirzebruch,  
  {\em Characteristic classes and homogeneous spaces,}  
  Amer. J. Math., {\bf 80} (1958), 964\ts--1029.

\bibitem[B]{B}   N. Bourbaki,  
  {\em Groupes et alg\`ebres de Lie,} IV--VI, 1968.

\bibitem[C]{C} C. Chevalley, {\em  Sur certains groupes simples,} 
Tohoku Math. J. (2) {\bf 7} (1955), 14-66. 



\bibitem[F]{F}   H. Freudenthal and H. de Vries,  
  {\em Linear Lie groups,} Pure and Applied Math., {\bf 35} Academic Press, 1969,  New York,.

\bibitem[KM]{KM}  A. Kono - M. Mimura,
  {\em Cohomology mod 3 of the classifying space
      of the Lie group $E_6$,}  Math. Scand., {\bf 46}(1980), 
      223\ts--235.


\bibitem[MS]{MS}  M. Mimura - Y. Sambe,
  {\em On the cohomology mod p of the classifying
      spaces of the exceptional Lie groups, I,}
  J. Math. of Kyoto Univ., {\bf 19} (1979), 553\ts--\ts581.

\bibitem[MST1]{MSTo} M. Mimura, Y. Sambe and M. Tezuka:
{\em Cohomology mod $3$ of the classifying spaces of
the exceptional Lie group of type $E_6$,} I, preprint (1986/1991).

\bibitem[MST2]{MSTe1} M. Mimura, Y. Sambe and M. Tezuka:
{\em Some remarks on the mod {\rm 3} cohomology of the classifying space of
the exceptional Lie group $E_6$, Proceedings of Workshop in Pure Mamthematics, 
Part {\rm III}, {\bf 17}}(1998), 139 - 159.

\bibitem[MST3]{MSTe2} M. Mimura, Y. Sambe and M. Tezuka:
{\em Cohomology mod $3$ of the classifying spaces of
the exceptional Lie group $E_6$,} I: structure of Cotor, arXiv:1112.5811.


\bibitem[N]{N} M. Nagata: {\em Theory of commutative fields.} 
 Translated from the 1985 Japanese edition by the author. 
 Translations of Mathematical Monographs, 125. American Mathematical Society, Providence, RI, 1993. 

\bibitem[RS1]{RS1}  M. Rothenberg - N. E. Steenrod,
  {\em The cohomology of the classifying spaces of H-spaces,} 
  Bull. AMS, {\bf 71} (1965), 872\ts--\ts875.

\bibitem[RS2]{RS2} M. Rothenberg - N. E. Steenrod,
  {\em The cohomology of the classifying spaces of H-spaces,}
    (mimeographed notes).
\bibitem[S]{LS} Larry Smith, 
{\em Polynomial invariants of finite groups.} 
Research Notes in Mathematics, {\bf 6}. 
A K Peters, Ltd., Wellesley, MA, 1995. 

\bibitem[T]{T}   H. Toda,
  {\em Cohomology of the classifying space of exceptional Lie groups,}
    Manifolds - Tokyo 1973, 265\ts--271.

\bibitem[TW]{TW}  H. Toda and T. Watanabe,
  {\em The integral cohomology ring of $F_4/T$  and  $E_6/T$,}
    J. Math. of Kyoto Univ., {\bf 14} (1974), 257--286.

\bibitem[V]{V}   H. O. Singh Varma,
  {\em The topology of $EI\! I\! I$ and a conjecture of
      Atiyah and Hirzebruch,}
   Nederl. Akad. Wet. Indag. Math., {\bf 30} (1968),
      67--71.


\end{thebibliography}
\end{document}